\documentclass{amsart}

\usepackage{graphicx}
\usepackage{dcolumn}
\usepackage{bm}
\usepackage{amsmath,amsthm,amssymb,stmaryrd}
\usepackage[cmtip,all]{xy}
\newcommand{\longsquiggly}{\xymatrix{{}\ar@{~>}[r]&{}}}
\numberwithin{equation}{section}

\newtheorem{theorem}{Theorem}[section]
\newtheorem{lemma}{Lemma}[section]
\newtheorem{corollary}{Corollary}[section]
\newtheorem{proposition}{Proposition}[section]
\newtheorem{definition}{Definition}[section]

\newcommand{\bbint}[2]{\ensuremath{\;\backslash\!\!\!\!\backslash\!\!\!\!\!\int_{#1}^{#2}}}

\newcommand{\floor}[1]{\lfloor #1 \rfloor}
\usepackage{accents}
\newlength{\dhatheight}

\newcommand{\reglim}[2]{\ensuremath{\overset{\times}{\underset{#1\rightarrow #2}{\lim}}}}
\newcommand{\hypergeom}[5]{\ensuremath{{}_{#1}F_{#2}}\!\!\left.\left(\begin{array}{c}
		#3   \\
		#4 
	\end{array}\right|#5\right)}

\usepackage{amsaddr}

\begin{document}
	\title[Finite Part Integration]{Regularized Limit, analytic continuation and finite-part integration}
	
	\author{Eric A. Galapon}
	\address{Theoretical Physics Group, National Institute of Physics, University of the Philippines, Diliman Quezon City, 1101 Philippines}
	\email{eagalapon@up.edu.ph}
	\date{\today}

%
%

\begin{abstract} Finite-part integration is a recent method of evaluating a convergent integral in terms of the finite-parts of divergent integrals deliberately induced from the convergent integral itself [E. A. Galapon, Proc. R. Soc., A 473, 20160567 (2017)]. Within the context of finite-part integration of the Stieltjes transform of functions with logarithmic growths at the origin, the relationship is established between the analytic continuation of the Mellin transform and the finite-part of the resulting divergent integral when the Mellin integral is extended beyond its strip of analyticity.   It is settled that the analytic continuation and the finite-part integral coincide at the regular points of the analytic continuation. To establish the connection between the two at the isolated singularities of the analytic continuation, the concept of regularized limit is introduced to replace the usual concept of limit due to Cauchy when the later leads to a division by zero. It is then shown that the regularized limit of the analytic continuation at its isolated singularities equals the finite-part integrals at the singularities themselves. The treatment gives the exact evaluation of the Stieltjes transform in terms of finite-part integrals and yields the dominant asymptotic behavior of the transform for arbitrarily small values of the parameter in the presence of arbitrary logarithmic singularities at the origin. 
\end{abstract}
\maketitle

\section{Introduction}
Divergent integrals may arise from a well-defined integral when the integrand is expanded and the resulting series integrated term by term without the required uniformity conditions for the interchange of the order of integration and summation to be valid. For example, an attempt to evaluate the Stieltjes transform
\begin{equation}\label{stst}
	\int_0^a \frac{f(t)}{(\omega+t)}\,\mathrm{d}t
\end{equation}
by binomially expanding the kernel about $\omega=0$,
\begin{equation}\label{bino}
		\frac{1}{(\omega+t)}=\sum_{k=0}^{\infty} (-1)^k \frac{\omega^k}{t^{k+1}},
\end{equation}
inside the integral and then distributing the integration in the series lead to the infinite series
\begin{equation}\label{seriesmo}
	\sum_{k=0}^{\infty} (-1)^k \omega^k \int_0^{a}\frac{f(t)}{t^{k+1}}\,\mathrm{d}t .
\end{equation}
If $f(t)$ is analytic at the origin, the series degenerates into an infinite series of divergent integrals. One may attempt at assigning values to the induced divergent integrals to give meaning to the series \eqref{seriesmo}. However, it is known that doing so, say, by analytic continuation or by finite-parts, may only partially reproduce some terms of the actual value of the integral. This problem is known as the problem of missing terms arising from term by term integration involving divergent integrals \cite{wong,wong2,wong3,galapon2}. 

McClure and Wong were the first to give a systematic solution to the problem of missing terms in the context of the asymptotic evaluation of the Stieltjes transform \eqref{stst} using the distributional approach where the divergent integrals are associated with distributions over a certain test function space \cite{wong3}. In particular, they associated the singular functions $t^{-s-\alpha}$ and $t^{-s-1}$ with the distributions
\begin{equation}
 \left<t^{-s-\alpha},\phi\right> = \frac{1}{(\alpha)_s}\int_{0}^{\infty} t^{-\alpha}\phi^{(s)}(t)\,\mathrm{d}t ,
\end{equation} 
\begin{equation}
\left<t^{-s-1},\phi\right> = -\frac{1}{s!}\int_{0}^{\infty} \phi^{(s)}(t)\ln t\,\mathrm{d}t ,
\end{equation} 
for non-negative integer $s$ and positive $\alpha<1$. By establishing the relationship between these distributions with the remainder terms in their expansions of the Stieltjes transform, they were able to obtain the missing terms. A feature of their work is the liberal use of the analytic continuation of the Mellin transform to assign values to integrals that would otherwise diverge. The method has been extended to allow asymptotic evaluation of the Stieltjes transform for algebraically decaying function \cite{lopez}, and it has been applied in deriving spectacularly accurate non-Poincar\'e type asymptotic expansions \cite{kaye}. 

Recently we revisited the problem of missing terms in the context of the exact evaluation of the Stieltjes transform \eqref{stst} using term by term integration involving divergent integrals \cite{galapon2}. Our effort there was motivated by our earlier observation that the Cauchy principal value and the Hadamard finite-part integral assumed contour integral representations  in the complex plane \cite{galapon1}. This led to the idea of recasting the problem of evaluating the Stieltjes transform \eqref{stst} in the complex plane where the finite-part of the divergent integrals induced by naive term by term integration could assume a contour integral representation. It was indeed the case that the finite-part integrals of the divergent integrals in the expansion \eqref{seriesmo} assumed a contour integral representation. This led to the definite assignment of the divergent integrals with values equal to their finite-parts and the identification of the missing terms as contributions coming from the singularity of the kernel of transformation. We referred to the method of evaluating a well-defined integral, such as the Stieltjes transform, in terms of the finite-parts of divergent integrals through their complex contour integral representations as finite-part integration \cite{galapon2}. 

We have since applied finite-part integration in evaluating the generalized Stieltjes transform and to date we have the result \cite{galapon3,galapon4,galapon5},
\begin{equation}\label{st}
	\int_0^a \frac{f(t)}{(\omega+t)^{\rho}}\mathrm{d}t = 	\sum_{k=0}^{\infty} \binom{-\rho}{k} \omega^k \bbint{0}{a}\frac{f(t)}{t^{k+\rho}}\,\mathrm{d}t + \Delta_{sc},\; \rho\geq 1,\;0<a\leq\infty,
\end{equation}
for complex analytic $f(t)$ in the interval $[0,a)$; the integral $\bbint{0}{a}t^{-k-\rho} f(t)\mathrm{d}t$ is the finite-part of the divergent integral $\int_{0}^{a}t^{-k-\rho} f(t)\mathrm{d}t$, $f(0)\neq 0$; $\Delta_{sc}$ is the recovered missing term and is traceable to the singularity of the kernel $(\omega+t)^{-\rho}$ in the complex plane which is at $t=-\omega$; the radius of convergence of the infinite series with respect to $\omega$  is the smaller of the upper limit of integration $a$ and the distance of the singularity of $f(t)$ nearest to the origin. Equation \eqref{st} represents an exact evaluation of the Stieltjes transform in terms of finite-parts of divergent integrals; moreover, it gives the dominant asymptotic behavior of the transform for arbitrarily small $\omega$ which is precisely described by the term $\Delta_{sc}$ because the infinite series of divergent integrals is $O(1)$ as $\omega\rightarrow 0$. Equation \eqref{st} has been used in obtaining new representations of various special functions, including sharper asymptotic behaviors of them \cite{galapon3,galapon4}, and new identities involving them \cite{galapon5}.

Let us recall the definition of the finite-part integral arising from the evaluation of the Stieltjes transform. If a locally integrable function $f(t)$ or any of its derivatives does not vanish at the origin, the integral
\begin{equation}\label{div}
	\int_0^a \frac{f(t)}{t^{\lambda}}\,\mathrm{d}t
\end{equation} 
is divergent for sufficiently large real part of $\lambda$ due to a non-integrable singularity at the origin. The finite-part of the divergent integral \eqref{div} is obtained by replacing the lower limit of integration by some small $\epsilon>0$ and decomposing the now convergent integral into the form
\begin{equation}\label{decomposition}
	\int_{\epsilon}^a \frac{f(t)}{t^{\lambda}}\, \mathrm{d}t = C_{\epsilon}+D_{\epsilon},
\end{equation}
where $C_{\epsilon}$ is the group of terms that converge in the limit as $\epsilon\rightarrow 0$, and $D_{\epsilon}$ is the group of terms that diverge in the same limit \cite{monegato,fox,hadamard}. The diverging terms, $D_{\epsilon}$, must consist only of diverging terms in algebraic powers of $\epsilon$ and $\ln\epsilon$. This requires that $f(t)$ can be developed as a Taylor expansion at the origin. 

The finite-part of the divergent integral is defined as the value of the limit of $C_{\epsilon}$ as $\epsilon$ approaches zero,
\begin{equation}\label{fpidef}
	\begin{split}
		\bbint{0}{a} \frac{f(t)}{t^{\lambda}}\,\mathrm{d}t = \lim_{\epsilon\rightarrow 0}C_{\epsilon}.
	\end{split}
\end{equation}
Equivalently, the finite-part is given by
\begin{equation}
	\bbint{0}{a}\frac{f(t)}{t^{\lambda}}\,\mathrm{d}t = \lim_{\epsilon\rightarrow 0}\left(\int_{\epsilon}^a \frac{f(t)}{t^{\lambda}}\,\mathrm{d}t-D_{\epsilon}\right), 
\end{equation}
following from the decomposition \eqref{decomposition} and the definition of the finite-part \eqref{fpidef}. When $a=\infty$ the finite-part is defined by the limit
\begin{equation}\label{fpidef2}
	\bbint{0}{\infty} \frac{f(t)}{t^{\lambda}}\,\mathrm{d}t = \lim_{a\rightarrow\infty} \bbint{0}{a} \frac{f(t)}{t^{\lambda}}\,\mathrm{d}t ,
\end{equation}
provided the limit exists. By definition the finite-part always exists and is finite. The definition does not exclude the possibility that the finite-part is equal to zero.  

In this paper we seek to establish the relationship between the finite-part integral
\begin{equation}
	\bbint{0}{a} \frac{f(t)}{t^{\lambda}}\,\mathrm{d}t ,\;\;\; \mathrm{Re}(\lambda)\geq 1, \;\;\; f(0)\neq 0 ,
\end{equation} and the Mellin transform of the function $f(t)$,
\begin{equation}\label{mt}
	\mathcal{M}_a[f(t);1-\lambda]=\int_0^a \frac{f(t)}{t^{\lambda}}\,\mathrm{d}t, \;\;\; d<\mathrm{Re}(\lambda)<1
\end{equation}
where $d<\mathrm{Re}(\lambda)<1$ is the strip of analyticity of the transform\footnote{By definition, the Mellin transform is an integration over the entire half-line, not in a finite interval. However, the integral \eqref{mt} can be interpreted as the Mellin transform of the function $f(t) H(a-t)$, where $H(x)$ is the Heaviside step function. For this reason, we refer to \eqref{mt} as a Mellin transform.}. Since the domains of the finite-part integral and the Mellin transform are disjoint, the relationship is to be established through the analytic continuation  of the Mellin transform, which we denote by
\begin{equation}
	\mathcal{M}_a^*[f(t);1-\lambda].
\end{equation} 
Examples exist in which the value of the analytic continuation and the finite part coincide at certain regular points of the analytic continuation. However, it is not clear that the equality is a general feature of the analytic continuation. One feature of the analytic continuation is that it may develop poles at the points where the Mellin integral is divergent; however, the finite-part always exists and is finite. In such a situation the value of the analytic continuation and the finite-part are not equal. The former is infinite or undefined while the later is finite. This indicates that the relationship between the analytic continuation and the finite-part is not straightforward.

We will study the relationship in the context of the finite-part integration of the Stieltjes integral 
\begin{equation}\label{probhere}
	\int_0^a\frac{h(t)}{(\omega+t)}\,\mathrm{d}t 
\end{equation}
where  
\begin{equation}\label{probhere2}
	h(t)=\sum_{k=0}^{\infty} \sum_{l=0}^{M(k)} k_{kl}(t) t^{-\nu_k} \ln^l t 
\end{equation}
in which the $k_{kl}(t)$'s are complex analytic in the interval of integration and the series converging uniformly in the same interval. Distribution of the integration leads to Stieltjes transforms of the form
\begin{equation}\label{components}
	\int_0^a \frac{k(t) \ln^n t}{t^{\nu}(\omega + t)}\,\mathrm{d}t ,\;\; 0\leq \mathrm{Re}(\nu)<1 .
\end{equation}
for $n=0, 1, 2,\dots$. And substituting the expansion \eqref{bino} back into \eqref{components} and performing term by term integration leads to the consideration of the divergent integrals
\begin{equation}\label{etona}
	\int_0^a \frac{k(t) \ln^n t}{t^{k+\nu+1}}\,\mathrm{d}t,
\end{equation}
for non-negative integer $k$. We will evaluate \eqref{probhere} through the Stieltjes transforms \eqref{components} by means of the finite-parts of the divergent integrals \eqref{etona}.

In the process, we establish the following results. First, under certain specific conditions to be specified later, we establish that the finite-part integral
\begin{equation}
	\bbint{0}{a} \frac{k(t) \ln^n}{t^{\lambda}}\, \mathrm{d}t,\;\; k(0)\neq 0,\;\;\mathrm{Re}(\lambda)\geq 1
\end{equation}
for all non-negative integer $n$ is completely determined by the analytic continuation of the Mellin transform 
\begin{equation}\label{mtmt}
	\int_0^a \frac{k(t)}{t^{\lambda}}\,\mathrm{d}t .
\end{equation}   
Second, we show that the explicit evaluation of the component Stieltjes transforms \eqref{components} for every positive integer $n$ in the full range of the parameter $\nu$ can be generated from the evaluation of the Stieltjes transform \eqref{stst}. To accomplish these we will need to introduce a generalization of the Cauchy limit that allows meaningful assignment of values to a holomorphic function at its isolated singularities. We will refer to the generalization as the regularized limit. The regularized limit will allow us to take meaningful limit where the Cauchy limit leads to a division by zero.  

We will show that repeated differentiation and application of the regularized limit leads to exact evaluation of the Stieltjes transform \eqref{components} in terms of finite-parts of the divergent integrals \eqref{etona}. Moreover, the results yield explicit expressions for the missing terms which carry the dominant behavior of the Stieltjes transform as $\omega$ approaches zero; for example, we will find the following asymptotic relations
\begin{equation}\label{asy1}
	\int_0^a \frac{k(t)\ln t}{t^{\nu} (\omega+t)}\,\mathrm{d}t \sim \frac{\pi k(0)}{\omega^{\nu}} \left(\pi \cot(\pi\nu)+ \ln\omega\right),\;\; \omega\rightarrow 0,
\end{equation} 
\begin{equation}\label{asy2}
	\int_0^a \frac{k(t)\ln t}{ (\omega+t)}\,\mathrm{d}t \sim -\frac{k(0)}{2} \ln^2\omega,\;\; \omega\rightarrow 0,
\end{equation} 
for all $k(t)$ that are complex analytic in the interval of integration and $k(0)\neq 0$ and for $\omega>0$. Observe that \eqref{asy2} cannot be continuously obtained from \eqref{asy1} using the usual Cauchy limit as $\nu\rightarrow 0$ because the right hand side would diverge; however, we will find that \eqref{asy2} is the regularized limit of \eqref{asy1}.

The rest of the paper is organized in three Parts. 

Part-I constitutes Sections-\ref{regularizedlim} and-\ref{genhospital}. In Section-\ref{regularizedlim} we define the regularized limit and establish its relationship with the usual notion of limit due to Cauchy. There we establish its most important property of being linear, a property not shared by the Cauchy limit. Much of our development will depend on this linear property of the regularized limit. In Section-\ref{genhospital} we obtain the explicit formula for the regularized limit of rational functions leading to division by zero with respect to the Cauchy limit. We will arrive at an expression generalizing the well-known L'Hospital's rule. 

Part-II constitutes Sections-\ref{nonlog} and-\ref{powerlog}. In Section-\ref{nonlog} we revisit the non-logarithmic case and obtain the relationship between the analytic continuation of the Mellin transform and the finite-part integral. There we give a first principle derivation of the contour integral representations of the finite-part integrals only intuited in \cite{galapon2}. In Section-\ref{powerlog} we establish the relationship between the analytic continuation of the Mellin transform in the presence of power logarithmic singularities and the finite-part integral. We will find that the finite-part integral in the full range of $\nu$ can be obtained from the contour integral representation of the finite-part integral for the non-logarithmic case. 

Part-III constitutes Sections-\ref{nonlogcase}, \ref{logcase} and \ref{combination}. In Section-\ref{nonlogcase} we revisit the finite-part integration of the Stieltjes transform for the non-logarithmic case considered in \cite{galapon2}. There we will rederive our results in a completely different manner based on analytic continuation. In Section-\ref{logcase} we evaluate the Stieltjes transform involving arbitrary integer power of the logarithm given by \eqref{components}, completing the ingredients in evaluating the Stieltjes transform for equation \eqref{probhere}. Finally, in Section-\ref{combination} we demonstrate the evaluation of the Stieltjes transform for a class of functions involving linear logarithmic singularity at the origin.

In this paper, $\log z$ denotes the complex logarithm  which is given by $\log z=\ln|z|+\arg{z}$, where $0\leq \arg{z}<2\pi$; $\log z$ coincides with $\ln t$ on top of the positive real axis. On the other hand, $\operatorname{Log}z$ denotes the complex principal value logarithm which is given by $\operatorname{Log}z=\ln|z|+\operatorname{Arg}z$, where $-\pi< \operatorname{Arg}z\leq \pi$; $\operatorname{Log} z$ analytically continues the natural logarithm $\ln t$ away from the positive real axis into the complex plane.

\section*{Part I}
\section{The Regularized Limit at an Isolated Singularity}\label{regularizedlim}

In the development to follow, we will face the problem of extracting meaningful information from the analytic continuation of the Mellin transform at its isolated singularities. Except at removable singularities, the usual limit due to Cauchy yields infinities or is undefined at isolated singularities. In the case of poles, the limit reduces to a division by zero. In this Section, we generalize the concept of the Cauchy limit with the generalization possessing the properties that it assigns a value equal to the Cauchy limit at regular points and removable singularities, and that it assigns a unique finite value at poles and essential singularities.  

Let $w(\lambda)$ be a function of the complex variable $\lambda$ analytic in some domain $D\subseteq\mathbb{C}$ except at some isolated points interior to $D$. If $\lambda_0$ is an interior point of $D$, then either $\lambda_0$ is a regular point or an isolated singularity of $w(\lambda)$. If $\lambda_0$ is an isolated singularity of $w(\lambda)$ in $D$, removable or essential or pole, then it is known that there exists a deleted neighborhood $\delta_{\lambda_0}=D(r,\lambda_0)\setminus\{\lambda_0\}$, where $D(r,\lambda_0)$ is an open disk of radius $r$ centered at $\lambda_0$, such that $w(\lambda)$ assumes the Laurent series expansion
\begin{equation}
w(\lambda)=\sum_{k=-\infty}^{\infty} a_k (\lambda-\lambda_0)^k ,
\end{equation}
for all $\lambda$ in $\delta_{\lambda_0}$. The radius $r$ is bounded by  the distance to the nearest singularity of $w(\lambda)$ from $\lambda_0$. The $a_k$'s are constants independent of $\lambda$ and are given by 
\begin{equation}\label{integrep}
a_k= \frac{1}{2\pi i}\oint_{|\lambda-\lambda_0|=\rho} \frac{w(\lambda)}{(\lambda-\lambda_0)^{k+1}}\,\mathrm{d}\lambda 
\end{equation}
where $\rho<r$. 
If all the coefficients for the negative powers of $(\lambda-\lambda_0)$ vanish, $a_{-1}=a_{-2}=\dots = 0$, the singularity $\lambda_0$ is a removable singularity; otherwise, $\lambda_0$ is either  an essential singularity or a pole. If $\lambda_0$ is an essential singularity, there are infinitely many negative powers in the expansion; if it is a pole, there are only finite number of negative powers and the largest power is the order of the pole.

The Laurent series can be decomposed into its principal part, which diverges as $\lambda\rightarrow\lambda_0$, and its regular part, which converges in the same limit. We designate them by $w_P(\lambda|\lambda_0)$ and $w_R(\lambda|\lambda_0)$, respectively, and are given by
\begin{equation}\label{pp}
w_P(\lambda|\lambda_0)=\sum_{k=1}^{\infty} \frac{a_{-k}}{ (\lambda-\lambda_0)^k}, 
\end{equation}
\begin{equation}\label{rp}
w_R(\lambda|\lambda_0)=\sum_{k=0}^{\infty} a_{k} (\lambda-\lambda_0)^k
\end{equation}
both converging for all $\lambda$ in $\delta_{\lambda_0}$. Within the deleted neighborhood, $w(\lambda)$ assumes the decomposition
\begin{equation}
    w(\lambda)=w_P(\lambda|\lambda_0) + w_R(\lambda|\lambda_0), \;\;\; \lambda \in \delta_{\lambda_0}.
\end{equation}

Now we introduce the concept of regularized limit of an analytic function at its isolated singularity.
\begin{definition}
	Let $\lambda_0$ be an isolated singularity of the function $w(\lambda)$. The regularized limit of the function $w(\lambda)$ as $\lambda\rightarrow\lambda_0$, to be denoted by
	\begin{equation}
	\reglim{\lambda}{\lambda_0}w(\lambda), \nonumber 
	\end{equation}
	 is defined by
	\begin{equation}
	\reglim{\lambda}{\lambda_0} w(\lambda) = \lim_{\lambda\rightarrow\lambda_0} \left[w(\lambda)-w_P(\lambda|\lambda_0)\right],
	\end{equation}
	where the limit in the right hand side is the limit in the usual sense of Cauchy. 
\end{definition} 
\begin{theorem}\label{representations}
	Let $\lambda_0$ be an isolated singularity of $w(\lambda)$. Then the regularized limit at $\lambda_0$ always exists and is unique. The value is given by 
	\begin{equation}\label{anaught}
	\reglim{\lambda}{\lambda_0} w(\lambda) =a_0
	\end{equation}
	where $a_0$ is the constant term in the regular-part,  $w_R(\lambda|\lambda_0)$, of $w(\lambda)$. Moreover, it admits the contour integral representation
	\begin{equation}\label{crep}
		\reglim{\lambda}{\lambda_0} w(\lambda)=\frac{1}{2\pi i} \oint_{|\lambda-\lambda_0|=\rho} \frac{w(\lambda)}{(\lambda-\lambda_0)}\, \mathrm{d}\lambda 
	\end{equation}
for sufficiently small $\rho$.
\end{theorem}
\begin{proof}
	In a deleted neighborhood of $\lambda_0$, $w(\lambda)-w_P(\lambda)|\lambda_0)=w_R(\lambda|\lambda_0)$. Then $\lim_{\lambda\rightarrow\lambda_0} \left[w(\lambda)-w_P(\lambda|\lambda_0)\right]=\lim_{\lambda\rightarrow\lambda_0}w_R(\lambda|\lambda_0)=a_0$. The integral representation follows directly from the integral representation of the coefficients of the Laurent expansion of $w(\lambda)$ and corresponds to $k=0$ in equation \eqref{integrep}. 
\end{proof}

It is important to emphasize that the definition of the regularized limit requires the Laurent expansion in a deleted neighborhood of the isolated singularity. The reason is that the Laurent expansion of a function is not unique, with the different expansions corresponding to different annuli centered at the singularity bounded by circles with radii given by the distances of two consecutive isolated singularities. For example, the function
$1/\lambda(\lambda+1)(\lambda+2)$ has the Laurent expansion
\begin{equation}\label{exp1}
	\frac{1}{\lambda(\lambda+1)(\lambda+2)}=-\frac{1}{2\lambda}+\sum_{k=2}^{\infty} \frac{(-1)^k}{\lambda^{k}}+\sum_{k=0}^{\infty} (-1)^k \frac{\lambda^{k}}{2^{k+2}}
\end{equation}
which is only valid in the annulus $1<|\lambda|<2$. Disregarding the annulus of analyticity of the expansion, the expansion gives the impression that $\lambda=0$ is an essential singularity which is contrary to the fact that the function has only a simple pole at the origin. Also it gives the regularized limit equal to $-1/4$. On the other hand, in the neighborhood of $\lambda=0$, we have the Laurent expansion
\begin{equation}
	\frac{1}{\lambda(\lambda+1)(\lambda+2)}=\frac{1}{2\lambda}-\sum_{k=0}^{\infty}(-1)^k \frac{(2^{k+2}-1)}{2^{k+2}}\lambda^{k}
\end{equation}
which is valid for all $0<|\lambda|<1$. This is the desired expansion in a deleted neighborhood of the origin. This yields the correct value at the origin and it yields the regularized limit $-3/4$. This value is different from the value arising from the expansion \eqref{exp1}. 

The following regularized limit will be useful to us in establishing the relationship between the analytic continuation of the Mellin transform and the finite-part integral.
\begin{proposition}
	If $w(\lambda)$ is analytic at $\lambda=\lambda_0$, then
	\begin{equation}\label{coco}
		\reglim{\lambda}{\lambda_0}\frac{w(\lambda)}{(\lambda-\lambda_0)^n} = \frac{w^{(n)}(\lambda_0)}{n!} ,
	\end{equation}
for all positive integer $n$.
\end{proposition}
\begin{proof}
	We expand $w(\lambda)$ about $\lambda=\lambda_0$ and obtain the leading term of the regular part of $w(\lambda)/(\lambda-\lambda_0)^n$ which is just the left hand side of \eqref{coco}.
\end{proof}

A desired property of the regularized limit is that it generalizes the Cauchy limit and that it yields the same limit as Cauchy's when it exists. Indeed the regularized limit does  so and is well-defined where the Cauchy limit does not exist.
\begin{theorem} \label{equality0} 
	The equality
	\begin{equation}\label{equality}
	\reglim{\lambda}{\lambda_0} w(\lambda)=\lim_{\lambda\rightarrow\lambda_0}w(\lambda),
	\end{equation}
	holds if and only if $\lambda_0$ is a removable singularity or a regular point of $w(\lambda)$.
\end{theorem}
\begin{proof}
	If $\lambda_0$ is a removable singularity or a regular point of $w(\lambda)$, the principal part of the Laurent series expansion of $w(\lambda)$ in a deleted neighborhood of $\lambda_0$ is identically zero and its regular part is just its Taylor series expansion about $\lambda_0$. Then the equality of the limits in \eqref{equality} follows. On the other hand, if $\lambda_0$ is not removable so that it is either a pole or an essential singularity, the principal part of $w(\lambda)$ does not identically vanish and the value of the Cauchy limit of $w(\lambda)$ at $\lambda_0$ is either infinite or undefined,  while the regularized limit is finite and well defined there.  
\end{proof}

The equality afforded by Theorem-\ref{equality0} and the finiteness of the regularized limit everywhere motivates introducing the concept of a regularized version of an analytic function $w(\lambda)$.
\begin{definition}
	Let $w(\lambda)$ be analytic in some open domain $D$ except perhaps at some isolated points in $D$. The regularized-$w(\lambda)$, to be denoted by
	\begin{equation}
		\overset{\times}{w}(\lambda),
	\end{equation} 
is the function of the complex variable $\lambda$ given by
	\begin{equation}
	\overset{\times}{w}(\lambda)=\reglim{\lambda'}{\lambda}w(\lambda'),
\end{equation}
for all $\lambda$ in $D$. 
\end{definition}
\noindent As defined,  $w(\lambda)$ and its regularized version $\overset{\times}{w}(\lambda)$ differ on a set of measure zero. In particular, they are equal everywhere where $w(\lambda)$ is analytic and differ at the isolated non-removable singularities of $w(\lambda)$.  In Part-II of the paper, we shall establish that the finite-part integral, defined as a function in the complex plane, is the regularized version of the analytic continuation of the Mellin transform. 

In general the Cauchy limit cannot be distributed over a sum, i.e. it is not linear. On the other hand, the regularized limit is linear.
\begin{theorem}\label{twoterm}
	Let $\lambda_0$ be an isolated singularity of $w_1(\lambda)$ and $w_2(\lambda)$. Then
	\begin{equation}\label{linearity}
	\reglim{\lambda}{\lambda_0}[\alpha_1 w_1(\lambda)+\alpha_2 w_2(\lambda)] = \alpha_1\reglim{\lambda}{\lambda_0}w_1(\lambda)+\alpha_2 \reglim{\lambda}{\lambda_0}w_2(\lambda), 
	\end{equation}
	for all complex numbers $\alpha_1$ and $\alpha_2$. 
\end{theorem}
\begin{proof}
	Since $\lambda_0$ is an isolated singularity for both $w_1(\lambda)$ and $w_2(\lambda)$, there exists a common deleted neighborhood $\delta_{\lambda_0}$ of $\lambda_0$ where they are both analytic. Then for a closed $C$ entirely lying in $\delta_{\lambda_0}$ and encircling $\lambda_0$, the sum $(w_1+w_2)(\lambda)$ can be substituted in the right hand side of the contour integral representation \eqref{crep} for $w(\lambda)$ and the integration distributed, giving equation \eqref{linearity}. 
\end{proof}

A stronger version of the linearity of the regularized limit is given by the following result.
\begin{theorem}\label{sequence}
	Let $w_0(\lambda)$, $w_1(\lambda)$, $w_2(\lambda)$, $\dots$ be an infinite sequence of functions of the complex variable $\lambda$, all of which analytic except perhaps at some isolated points in some region $D$. Let $\lambda_0$ be an interior point of $D$, and let the infinite series $\sum_{k=0}^{\infty} w_k(\lambda)$ converge uniformly in some deleted neighborhood, $D(r,\lambda_0)$, of $\lambda_0$. Then
	\begin{equation}\label{sisi}
	\reglim{\lambda}{\lambda_0}\sum_{k=0}^{\infty} w_k(\lambda)=\sum_{k=0}^{\infty} \reglim{\lambda}{\lambda_0}w_k(\lambda) .
	\end{equation} 
\end{theorem}
\begin{proof} 
Applying the contour integral representation \eqref{crep} of the regularized limit on the left hand side of \eqref{sisi} and distributing the integration over the sum leads to the right hand side. The term by term integration of the infinite series is justified because the series converges uniformly in $D(r,\lambda_0)$.
\end{proof}
\noindent In Theorem-\ref{sequence}, $\lambda_0$ need not be a common isolated singularity of the $w_k(\lambda)$'s. Some of the elements of the sequence may be analytic at $\lambda_0$ or altogether analytic there.

A consequence of Theorem-\ref{equality0} is that we can replace the Cauchy limit wherever it appears with the regularized limit. The linearity of the later afforded by Theorems-\ref{twoterm} and-\ref{sequence} allows term by term application of the regularized limit which is generally not possible with the Cauchy limit. For example, we have by the usual limit
\begin{equation}\label{poy}
	\lim_{\lambda\rightarrow 0}\frac{e^{\lambda}-e^{-\lambda}}{\lambda}=2 .
\end{equation}
We cannot distribute the limit because doing so leads to a division by zero. However, we can replace the Cauchy limit with the regularized limit and appeal to the linearity of the later to have
\begin{equation}\label{qwe}
	\lim_{\lambda\rightarrow 0}\frac{e^{\lambda}-e^{-\lambda}}{\lambda}= \reglim{\lambda}{0}\frac{e^{\lambda}}{\lambda}-\reglim{\lambda}{0}\frac{e^{-\lambda}}{\lambda}.
\end{equation}
Both limits in the right hand side are well-defined and are given by
\begin{equation}
	\reglim{\lambda}{0}\frac{e^{\pm\lambda}}{\lambda}=\pm 1,
\end{equation} 
according to equation \eqref{coco}. Substituting these values back into equation \eqref{qwe}, we reproduce the result \eqref{poy}. We will see that much of the development to follow depends on the linear property of the regularized limit. 

\section{Generalized L'Hospital's Rule}\label{genhospital}
Later we will encounter the problem of evaluating the regularized limit of the rational function $f(\lambda)/g(\lambda)$ as $\lambda$ approaches some value $\lambda_0$, where $f(\lambda)$ and $g(\lambda)$ are both analytic at $\lambda_0$, with $f(\lambda_0)\neq 0$ and $g(\lambda_0)=0$. In this Section, we derive the explicit formula for the limit,
\begin{equation}\label{hehi}
	\reglim{\lambda}{\lambda_0}\frac{f(\lambda)}{g(\lambda)} 
\end{equation}
for arbitrary orders of zero of $g(\lambda)$. First, we prove the following representation of the regularized limit which will be the basis for our derivation of the desired formula for \eqref{hehi}.  

\begin{lemma}\label{representations2}
	Let $\lambda_0$ be an isolated singularity of the analytic function $w(\lambda)$.  If $w(\lambda)$ has a pole of order $n$ at $\lambda=\lambda_0$, then 
	\begin{equation}\label{prep}
	\reglim{\lambda}{\lambda_0} w(\lambda) = \frac{1}{n!}\lim_{\lambda\rightarrow \lambda_0}\frac{\mathrm{d}^n}{\mathrm{d}\lambda^n}\left[(\lambda-\lambda_0)^n w(\lambda)\right].
	\end{equation}
\end{lemma}
\begin{proof}
	In a deleted neighborhood of $\lambda_0$, the principal part of $w(\lambda)$ admits the expansion
	\begin{equation}\label{ff}
	w(\lambda)= \sum_{k=1}^n \frac{a_{-k}}{(\lambda-\lambda_0)^k} + a_0 + \sum_{k=1}^{\infty} \frac{a_k}{(\lambda-\lambda_0)^k},
	\end{equation}
	where $a_{-n}\neq 0$. Multiplying \eqref{ff} with $(\lambda-\lambda_0)^n$ and taking the $n$-th derivative with respect to $\lambda$, the contributions from the principal part vanishes and 
	\begin{equation}
	\frac{\mathrm{d}^n}{\mathrm{d}\lambda^n}\left[(\lambda-\lambda_0)^n w(\lambda)\right] = n!\, a_0 + O(\lambda-\lambda_0).
	\end{equation}
	Taking the limit $\lambda\rightarrow\lambda_0$ leads to equation \eqref{anaught} which establishes equation \eqref{prep}. 
\end{proof}

\begin{theorem}
	Let $f(\lambda)$ and $g(\lambda)$ be analytic at $\lambda_0$ with $f(\lambda_0)\neq 0$ and $g(\lambda_0)=0$. If $\lambda_0$ is a zero of $g(\lambda)$ of order $n$ so that $f(\lambda)/g(\lambda)$ has a pole of order $n$ at $\lambda_0$, then
	\begin{equation}\label{reglimit}
	\reglim{\lambda}{\lambda_0}\frac{f(\lambda)}{g(\lambda)} = \left(\frac{n!}{g^{(n)}(\lambda_0)}\right)^{n+1}\,\mathrm{det}\Delta^{(n)}(\lambda_0)
	\end{equation}
	where $\Delta^{(n)}(\lambda_0)$ is the $(n+1)\times(n+1)$ matrix given by
	\begin{equation}\label{determinant}
	\Delta^{(n)}(\lambda_0) = \left[\begin{array}{ccccc}
	\frac{g^{(n)}(\lambda_0)}{n!}  & 0 & 0 & \cdots & \frac{f(\lambda_0)}{0!} \\
	\frac{g^{(n+1)}(\lambda_0)}{(n+1)!}  & \frac{g^{(n)}(\lambda_0)}{n!}  & 0 & \cdots & \frac{f^{(1)}(\lambda_0)}{1!} \\
	\frac{g^{(n+2)}(\lambda_0)}{(n+2)!}  & \frac{g^{(n+1)}(\lambda_0)}{(n+1)!}  & \frac{g^{(n)}(\lambda_0)}{n!}  & \cdots & \frac{f^{(2)}(\lambda_0)}{2!} \\
	\vdots & \vdots & \vdots & \vdots & \vdots \\
	 \frac{g^{(2n)}(\lambda_0)}{(2n)!}  & \frac{g^{(2n-1)}(\lambda_0)}{(2n-1)!}  & \frac{g^{(2n-2)}(\lambda_0)}{(2n-2)!}  & \cdots & \frac{f^{(n)}(\lambda_0)}{n!} 
	\end{array}\right] .
	\end{equation}
	Equivalently 
	\begin{equation}\label{reglimit3}
	\reglim{\lambda}{\lambda_0}\frac{f(\lambda)}{g(\lambda)} = \frac{f^{(n)}(\lambda_0)}{g^{(n)}(\lambda_0)} + \left(\frac{n!}{g^{(n)}(\lambda_0)}\right)^{n+1}\, \sum_{k=0}^{n-1}\frac{f^{(k)}(\lambda_0)}{k!} C_{kn}(\lambda_0),
	\end{equation}
	where $C_{kn}(\lambda_0)$ is the cofactor of the element $f^{(k)}(\lambda_0)/k!$	of the matrix $\Delta^{(n)}(\lambda_0)$. 
\end{theorem}
\begin{proof}
	Since $f(\lambda)$ and $g(\lambda)$ are analytic at $\lambda_0$, they admit the expansions
	\begin{equation}\label{expandf}
	f(\lambda) = \sum_{k=0}^{\infty} \frac{f^{(k)}(\lambda_0)}{k!} (\lambda-\lambda_0)^k ,\;\; f(\lambda_0)\neq 0,
	\end{equation}
	\begin{equation}\label{expandg}
	g(\lambda) = \sum_{k=n}^{\infty} \frac{g^{(k)}(\lambda_0)}{k!} (\lambda-\lambda_0)^k,\;\; g^{(n)}(\lambda_0)\neq 0,
	\end{equation}
	for all $\lambda$ sufficiently close to $\lambda_0$. The expansion \eqref{expandg} for $g(\lambda)$ follows from the fact that $\lambda_0$ is a zero of $g(\lambda)$ of order $n$. Then
	\begin{eqnarray}\label{quotient}
	(\lambda-\lambda_0)^n \frac{f(\lambda)}{g(\lambda)} = \frac{\sum_{k=0}^{\infty} \frac{f^{(k)}(\lambda_0)}{k!} (\lambda-\lambda_0)^k}{\sum_{k=0}^{\infty} \frac{g^{(k+n)}(\lambda_0)}{(k+n)!} (\lambda-\lambda_0)^k} ,
	\end{eqnarray} 
	in a deleted neighborhood of $\lambda_0$.
	
	We now wish to obtain the expansion of the right hand side of \eqref{quotient} about $\lambda_0$ or obtain the quotient of the indicated division of the two infinite series. The following result is known \cite[pg. 436]{markus}:
	\begin{equation}\label{quotient1}
	\frac{\sum_{k=0}^{\infty}a_k (z-z_0)^k}{\sum_{k=0}^{\infty} b_k (z-z_0)^k} = \sum_{k=0}^{\infty} c_k (z-z_0)^k,
	\end{equation}
	where
	\begin{equation}\label{quotient2}
	c_k = \frac{1}{b_0^{k+1}} \, \mathrm{det} \left[\begin{array}{ccccc}
	b_0 & 0 & 0 & \cdots & a_0\\
	b_1& b_0 & 0 & \cdots & a_1 \\
	b_2 & b_1& b_0 & \cdots & a_2\\
	\vdots & \vdots &\vdots &\vdots &\vdots \\
	b_k & b_{k-1} & b_{k-2} & \cdots& a_k
	\end{array}\right] , \;\; k = 0, 1, 2,\dots . 
	\end{equation}
	When equations \eqref{quotient1} and \eqref{quotient2} are applied to the quotient in \eqref{quotient}, the coefficients are identified to be
	\begin{equation}\label{coeffs}
	a_k = \frac{f^{(k)}(\lambda_0)}{k!}, \;\; b_k=\frac{g^{(n)}(\lambda_0)}{(n+k)!}, \;\; k=0, 1, 2, \dots 
	\end{equation}
	Substituting these coefficients back into equations \eqref{quotient1} and \eqref{quotient2} yields the desired quotient in \eqref{quotient}.  
	
	The regularized limit can now evaluated using using Lemma-\ref{representations2}. Performing the indicated differentiations, we obtain
	\begin{equation}
	\frac{1}{n!} \frac{\mathrm{d}^n}{\mathrm{d}\lambda^n} \left((\lambda-\lambda_0)^n\frac{f(\lambda)}{g(\lambda)}\right) = c_n + O(\lambda-\lambda_0).
	\end{equation}  
	Taking the limit $\lambda\rightarrow\lambda_0$ yields the regularized limit of $f(\lambda)/g(\lambda)$ at $\lambda_0$, which is given by 
	\begin{equation}\label{reglimit2}
	\reglim{\lambda}{\lambda_0}\frac{f(\lambda)}{g(\lambda)} = c_n .
	\end{equation}
	Substituting the coefficients \eqref{coeffs}  into equation \eqref{quotient2} with $k=n$ gives expression  \eqref{reglimit} for the regularized limit with the appropriate determinant \eqref{determinant}. 
	
	To show equation \eqref{reglimit3}, we expand the determinant of $\Delta^{(n)}(\lambda_0)$ along the last column,
	\begin{equation}
	\mathrm{det}\Delta^{(n)}(\lambda) = \sum_{k=0}^n \Delta_{k n}(\lambda_0) C_{k n}(\lambda_0),
	\end{equation}
	where $\Delta_{k n}^{(n)}(\lambda_0)$ is the $k$-th row and $n$-th column element of the matrix $\Delta^{(n)}(\lambda_0)$ and $C_{k n}$ is the corresponding cofactor of $\Delta_{k n}^{(n)}(\lambda_0)$. The cofactor is given by  $C_{k n}=(-1)^{k+n} \mathrm{det}\Delta^{(n)}(\lambda_0)[k|n]$, where $\Delta^{(n)}(\lambda_0)[k|n]$ is the matrix obtained by removing the $k$-th row and $n$-th column of the matrix $\Delta^{(n)}(\lambda_0)$. Also the last column elements are $\Delta_{k n}^{(n)}(\lambda_0)=f^{(k)}(\lambda_0)/k!$. We have in particular
	\begin{equation}
	\Delta^{(n)}(\lambda_0)[n|n] = \left[\begin{array}{ccccc}
	\frac{g^{(n)}(\lambda_0)}{n!}  & 0 & 0 & \cdots & 0 \\
	\frac{g^{(n+1)}(\lambda_0)}{(n+1)!}  & \frac{g^{(n)}(\lambda_0)}{n!}  & 0 & \cdots & 0\\
	\frac{g^{(n+2)}(\lambda_0)}{(n+2)!}  & \frac{g^{(n+1)}(\lambda_0)}{(n+1)!}  & \frac{g^{(n)}(\lambda_0)}{n!}  & \cdots & 0 \\
	\vdots & \vdots & \vdots & \vdots & \vdots \\
	\frac{g^{(2n)}(\lambda_0)}{(2n)!}  & \frac{g^{(2n-1)}(\lambda_0)}{(2n-1)!}  & \frac{g^{(2n-2)}(\lambda_0)}{(2n-2)!}  & \cdots & \frac{g^{(n)}(\lambda_0)}{n!} 
	\end{array}\right] .
	\end{equation}
	Since the matrix is upper triangular, the determinant is just the product of the elements of the diagonal. Then we have the cofactor
	\begin{equation}
	C_{n n}(\lambda_0) = \left(\frac{g^{(n)}(\lambda)_0}{n!}\right)^n .
	\end{equation}
	Isolating the $n$-th term in the expansion of the determinant, the determinant assumes the form
	\begin{equation}
	\mathrm{det}\Delta^{(n)}(\lambda) = f^{(n)}(\lambda_0)\frac{(g^{(n)})^n}{(n!)^{n+1}}+\sum_{k=0}^{n-1} \Delta_{k n}(\lambda_0) C_{k n}(\lambda_0),
	\end{equation}
	Substituting this back into equation \eqref{reglimit}, we obtain the expression \eqref{reglimit3}.
\end{proof}

Regularized limits for specific values of order $n$ can now be obtained. Here we only list the specific cases of simple and double poles.
\begin{corollary} Let $f(z)$ and $g(z)$ be analytic at $\lambda_0$ with $f(\lambda_0)\neq 0$ and $g(\lambda_0)=0$. If $\lambda_0$ is a simple zero of $g(z)$, then 
\begin{equation}\label{reglimsimple}
\reglim{\lambda}{\lambda_0}\frac{f(\lambda)}{g(\lambda)} = \frac{f'(\lambda_0)}{g'(\lambda_0)} - \frac{f(\lambda_0) g''(\lambda_0)}{2 (g'(\lambda_0))^2}.
\end{equation}
\end{corollary}
\begin{proof}
For a simple pole, $n=1$, we have $b_0=g^{(1)}(\lambda_0)$ so that the coefficient $c_1$ is given by
\begin{equation}
c_1 = \frac{1}{(g^{(1)}(\lambda_0))^2}\, \mathrm{det}\left[\begin{array}{cc}
g^{(1)}(\lambda_0) & f(\lambda_0)\\
\frac{1}{2} g^{(2)}(\lambda_0) & f^{(1)}(\lambda_0)
\end{array}\right] .
\end{equation}
Evaluating the determinant, we obtain the regularized limit for simple poles given by equation \eqref{reglimsimple}.
\end{proof}
\begin{corollary} Let $f(z)$ and $g(z)$ be analytic at $\lambda_0$ with $f(\lambda_0)\neq 0$ and $g(\lambda_0)=0$. If $\lambda_0$ is a double zero of $g(z)$, then
\begin{equation}\label{reglimdouble}
	\begin{split}
\reglim{\lambda}{\lambda_0} \frac{f(\lambda)}{g(\lambda)} =& \frac{f''(\lambda_0)}{g''(\lambda_0)}  -\frac{2}{3} \frac{f'(\lambda_0) g'''(\lambda_0)}{(g''(\lambda_0))^2}\\
& + \frac{f(\lambda_0) \left(4 (g'''(\lambda_0))^2 - 3 g''(\lambda_0) g''''(\lambda_0)\right)}{18 (g''(\lambda_0))^3} . 
\end{split}
\end{equation}
\end{corollary}
\begin{proof}
For a double pole, $n=2$, we have $b_0=g^{(2)}(\lambda_0)/2$ so that the coefficient $c_2$ is given by
\begin{equation}
c_2= \frac{2^3}{(g^{(2)}(\lambda_0))^3}\,\mathrm{det} \left[\begin{array}{ccc}
\frac{1}{2!} g^{(2)}(\lambda_0) & 0 & f(\lambda_0)\\
\frac{1}{3!} g^{(3)}(\lambda_0) & \frac{1}{2!} g^{(2)}(\lambda_0) & f^{(1)}(\lambda_0)\\
\frac{1}{4!} g^{(4)}(\lambda_0) & \frac{1}{3!} g^{(3)}(\lambda_0) & \frac{1}{2!} f^{(2)}(\lambda_0)
\end{array}\right] .
\end{equation}
Evaluating the determinant, we obtain the regularized limit for double poles given by equation \eqref{reglimdouble}.
\end{proof}
\noindent Regularized limits for higher order poles can be obtained similarly from the general expression.

Observe that the first term of the regularized limit in \eqref{reglimit3} is the expected limit using L'Hospital rule when $f(\lambda)$ and $g(\lambda)$ happen to have the same orders of zero at $\lambda_0$. Indeed the regularized limit coincides with L'Hospital's rule when $f(\lambda)$ and $g(\lambda)$ happen to have the same order $n$ of zeros at $\lambda=\lambda_0$, 
\begin{equation}\label{hospital}
\reglim{\lambda}{\lambda_0} \frac{f(\lambda)}{g(\lambda)} = \lim_{\lambda\rightarrow\lambda_0}\frac{f(\lambda)}{g(\lambda)}= \frac{f^{(n)}(\lambda_0)}{g^{(n)}(\lambda_0)}.
\end{equation}
This shows that the regularized limit is a generalization of the L'Hospital's rule when the later leads to a division by zero in the sense of Cauchy limit. Observe though that \eqref{reglimit3} reduces to L'Hospital's rule when the derivatives of $g(\lambda)$ vanish at sufficiently high orders. For example, for simple poles, if $g''(\lambda_0)=0$ the regularized limit reduces to L'Hospitals limit \eqref{hospital} even though the higher derivatives do not vanish. For double poles, the regularized limit reduces to L'Hospital's limit if $g'''(\lambda_0)=0$ and $g''''(\lambda_0)=0$ even though the higher derivatives do not vanish.  In general, when $g^{(n+1)}(\lambda_0)=g^{(n+2)}(\lambda_0)=\cdots=g^{(2n)}(\lambda_0)=0$ for any positive integer $n$, with higher order derivatives not necessarily vanishing, the regularized limit reduces to \eqref{hospital}. 

In our application of the regularized-limit, we are presented with an analytic function $w(\lambda)$ that is known to have a pole singularity at some $\lambda=\lambda_0$ but not yet in the form $w(\lambda)=f(\lambda)/g(\lambda)$. For our result to be useful, we will have to write the given $w(\lambda)$ in rational form with $g(\lambda)$ explicitly vanishing at $\lambda_0$ and $f(\lambda)$ non-vanishing at the same point. However, there is not a single way of rationalizing $w(\lambda)$ and question arises whether the outcome is the same for all rationalizations of $w(\lambda)$. 
Given a function $w(\lambda)$ and its  particular rationalization $w(\lambda)=f(\lambda)/g(\lambda)$. 
Let $\Phi(\lambda)$ be analytic at $\lambda=\lambda_0$ and $\Phi(\lambda_0)\neq 0$. Define the new rationalization $w(\lambda)=\tilde{f}(\lambda)/\tilde{g}(\lambda)$, where $\tilde{f}(\lambda)=f(\lambda)\Phi(\lambda)$ and $\tilde{g}(\lambda)=g(\lambda)\Phi(\lambda)$. Substituting this new rationalization back into Lemma-\ref{representations2}, we obtain the equality,
\begin{eqnarray}
	\reglim{\lambda}{\lambda_0}\frac{\tilde{f}(\lambda)}{\tilde{g}(\lambda)} =\reglim{\lambda}{\lambda_0}\frac{f(\lambda)}{g(\lambda)},
	\end{eqnarray}
	so that both rationalizations yield the same regularized limits. In our application below, we will choose the most convenient rationalization. For example, for simple poles it maybe convenient to choose the rationalization that yields an expression for the finite-part which is similar to L'Hospital's rule, which is achieved when $g''(\lambda_0)=0$. 
	
	\begin{figure}
		\includegraphics[scale=0.5]{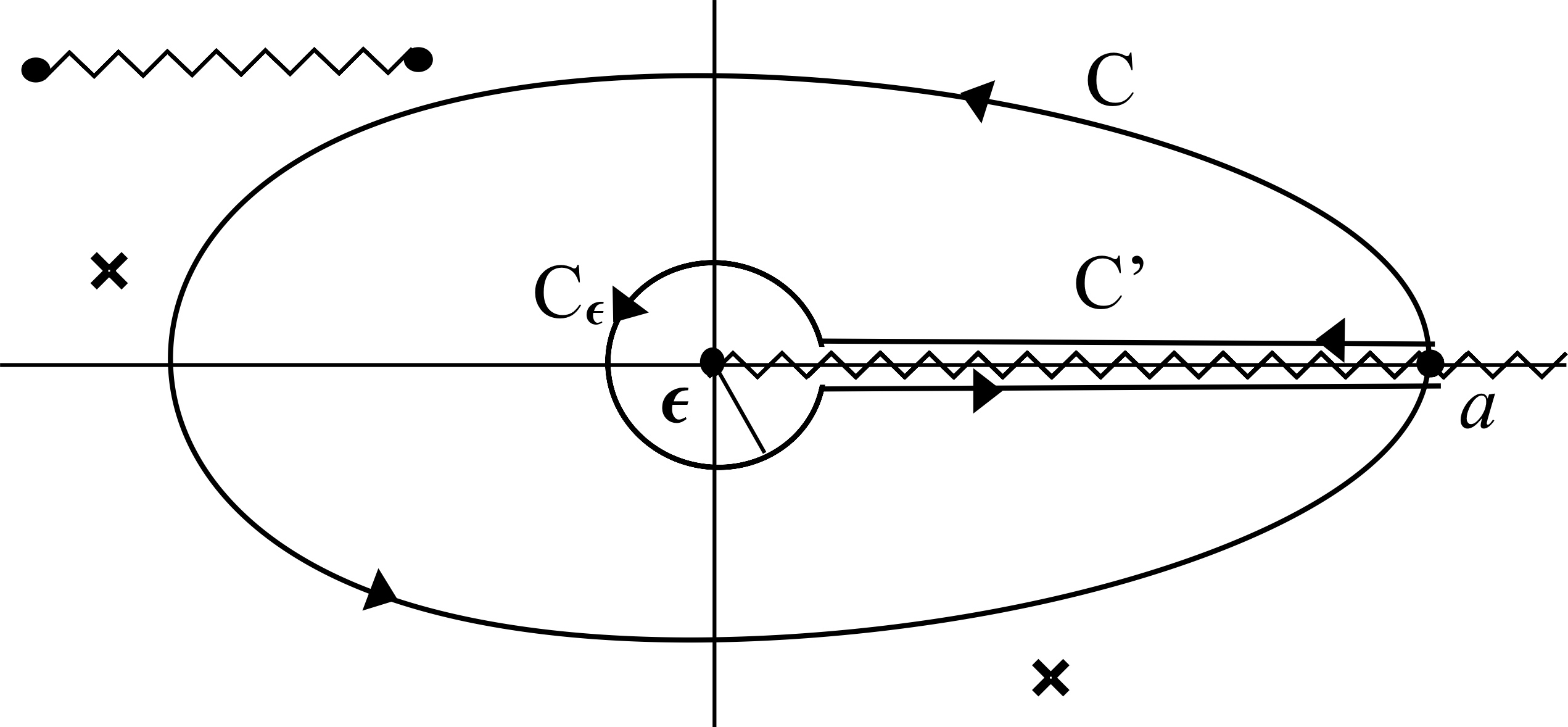}
		\caption{The contour $C$ of integration  in the construction of the analytic continuation of the Mellin transform. The contour $C$ does not enclose any pole or cross any branch cut of $k(z)$. }
		\label{boat1}
	\end{figure}
 
 \section*{Part II}
 \section{The Relationship Between Analytic Continuation and Finite-part Integral in the Absence of Logarithmic Singularities} \label{nonlog}
 We now address the first of the two main problems of the paper---to uncover the relationship between the Mellin transform and the finite-part integral. We limit our investigation to the class of functions which we denote by $\mathcal{K}_a$, where $a$ is a positive real number and may be infinite. 
 
 The elements of $\mathcal{K}_a$ are functions of the real variable $t$, $k(t)$, with the following properties: (1) $k(0)\neq 0$; (2) $k(t)$ has a complex extension, $k(z)$, which is analytic in the interval $[0,a)$ such that $k(t)$ is the restriction of $k(z)$ in the said interval; (3) the Mellin transform of $k(t)$ exists, 
 \begin{equation}\label{mt1}
 	\mathcal{M}_a[k(t);1-\lambda]= \int_0^a \frac{k(t)}{t^{\lambda}}\,\mathrm{d}t, \;\; d<\mathrm{Re}(\lambda)<1,
 \end{equation}
for some $d<1$.  
 If $a<\infty$, then $d=-\infty$; on the other hand, if $a=\infty$, $d$ may be finite or negative infinite depending on the behavior of $k(t)$ at infinity. By the principle of analytic continuation, the complex extension $k(z)$ is uniquely identified by the conditions imposed on it. The function $\cos^2(t)$ belongs to $\mathcal{K}_a$ for all $a<\infty$ but does not belong to $\mathcal{K}_{\infty}$ because $\int_0^{\infty} t^{-\lambda}\cos^2(t)\,\mathrm{d}t$ does not converge for all $\lambda$.  
 
 From hereon when ever we refer to ``$\mathcal{K}_a$'' we will always mean that the upper limit of integration $a$ may take on all possible positive numbers including infinity unless otherwise indicated. That is any result accompanied with the statement ``given $k(t)$ in $\mathcal{K}_a$'' or ``for all $k(t)$ in $\mathcal{K}_a$'' or any statement similar to them will mean that the result holds for either finite positive $a$ or infinite $a$, again, unless otherwise indicated. In specific situations where $a=\infty$, we will refer to the Mellin transform simply as $\mathcal{M}[k(t);1-\lambda]$. 
 
 In this Section, for every $k(t)$ in $\mathcal{K}_a$, we wish to establish the relationship between the finite part integral 
 \begin{equation}\label{fpi1}
 	\bbint{0}{a}\frac{k(t)}{t^{\lambda}}\,\mathrm{d}t ,\;\;\;\mathrm{Re}(\lambda)\geq 1,
 \end{equation}
 and the Mellin transform of $k(t)$ given by equation \eqref{mt1}. The domains of definition of the finite-part integral \eqref{fpi1} and the Mellin transform \eqref{mt1} are disjoint. If a relationship exists between the two, this relationship must be through the analytic continuation of the Mellin transform, $\mathcal{M}^*_a[k(t);1-\lambda]$, covering the domain of the finite-part integral.

 \subsection{Analytic Continuation}
 We now obtain the analytic continuation of the Mellin transform $\mathcal{M}_a[k(t);1-\lambda]$ in the domain of the finite-part integral, the region $\mathrm{Re}(\lambda)\geq 1$.
 Consider the contour integral
 \begin{equation}
 	\int_C \frac{k(z)}{z^{\lambda}}\,\mathrm{d}z,
 \end{equation}
 where $z^{\lambda}$ takes the positive real axis as its branch cut and coincides with  $t^{\lambda}$ for $t>0$ on top of the real line, and $C$ is the contour straddling the branch cut of $z^{-\lambda}$ which starts and ends at $a$ itself as depicted in Figure-\ref{boat1}; the contour $C$ does not enclose any pole or cross any branch cut of $k(z)$. We deform the contour $C$ into the contour $C'$ as depicted in Figure-\ref{boat1} to obtain
 \begin{equation}
 	\int_C \frac{k(z)}{z^{\lambda}}\,\mathrm{d}z = \left(e^{-2\pi\lambda i}-1\right) \int_{\epsilon}^a \frac{k(t)}{t^{\lambda}}\,\mathrm{d}t + \int_{C_{\epsilon}}\frac{k(z)}{z^{\lambda}}\,\mathrm{d}z ,
 \end{equation}
 where $C_{\epsilon}$ is a circular path of radius $\epsilon$. (From hereon, whenever we say ``deform the contour from $C$ to $C'$'' we will always mean the deformation indicated by Figure-\ref{boat1}.) For $\mathrm{Re}(\lambda)<1$, the integral along $C_{\epsilon}$ vanishes in the limit as $\epsilon\rightarrow 0$. Under this condition, we obtain
 \begin{equation}\label{ac1}
 	\int_0^a\frac{k(t)}{t^{\lambda}}\mathrm{d}t = \frac{1}{\left(e^{-2\pi\lambda i}-1\right)} \int_{C}\frac{k(z)}{z^{\lambda}}\,\mathrm{d}z, \;\;\; \mathrm{Re}(\lambda)<1 .
 \end{equation}
 
 The right hand side of equation \eqref{ac1} is analytic in the region $\mathrm{Re}(\lambda)\geq 1$ except at some isolated points. Thus the right hand side is the desired analytic continuation of the Mellin transform,
 \begin{equation}\label{analcon}
 	\mathcal{M}^*_a[k(t);1-\lambda] = \frac{1}{\left(e^{-2\pi\lambda i}-1\right)} \int_{C}\frac{k(z)}{z^{\lambda}}\,\mathrm{d}z .
 \end{equation}
The analytic continuation \eqref{analcon} extends $\mathcal{M}_a[k(t);1-\lambda]$ to the right of its strip of analyticity but not to its left. That is the domain of $\mathcal{M}_a^*[k(t);1-\lambda]$ is $-d<\mathrm{Re}(\lambda)$. In this domain, the analytic continuation is analytic everywhere except possibly at the points $\lambda=1, 2, 3, \dots$. Clearly, the points $\lambda=1, 2, 3, \dots$ are at most simple poles. For positive integer $m$, the integral around the contour reduces to an integration around a closed curve and it yields the value 
 \begin{equation}
 	\int_C \frac{k(z)}{z^m}\,\mathrm{d}z =\frac{ 2\pi i}{m!} k^{(m-1)}(0) .
 \end{equation}
 Since $\lambda=m$ is a simple zero of $(e^{-2\pi\lambda i}-1)$, $\lambda=m$ is a removable singularity when $k^{(m-1)}(0)=0$; but when $k^{(m-1)}(0)\neq 0$, $\lambda=m$ is a simple pole. At the removable singularities, we will assign $\mathcal{M}_a^*[k(t);1-\lambda]$ the value equal to its Cauchy limit there, rendering the analytic continuation analytic at those points. We summarize our result by the following statement.
 \begin{theorem}\label{one}
 	Let $k(t)$ be in $\mathcal{K}_a$ and $k(z)$ its complex extension. Then the analytic continuation of the Mellin transform $\mathcal{M}_a[k(t);1-\lambda]$ in the region $\mathrm{Re}(\lambda)\geq 1$ is given by
 	\begin{equation}
 		\mathcal{M}^*_a[k(t);1-\lambda] = \frac{1}{\left(e^{-2\pi\lambda i}-1\right)} \int_{C}\frac{k(z)}{z^{\lambda}}\,\mathrm{d}z,
 	\end{equation}
 	where ${C}$ is the contour straddling the branch cut of $z^{-\lambda}$ and starting from $a$ and ending at $a$ as well; moreover, ${C}$ does not enclose any of the singularities of $k(z)$. The analytic continuation has 
 	at most simple poles at $\lambda=1, 2, 3, \dots$, and analytic everywhere else in the region. When $k^{(m-1)}(0)=0$ for some positive integer $m$, $\lambda=m$ is a removable singularity of the analytic continuation; on the other hand, when $k^{(m-1)}(0)\neq 0$, $\lambda=m$ is a simple pole.  
 \end{theorem}
 
 \subsection{The Analytic Continuation and Finite-part Integrals}
 We now establish the relationship between the analytic continuation of the Mellin transform and the finite part of the same integral when evaluated beyond its strip of analyticity. There are two cases to consider, when $\lambda\neq m$ and when $\lambda=m$ for positive integer $m$. In the former, $\lambda$ is a regular point of the analytic continuation and in the later it is at most a simple pole. 
 
  \begin{lemma}\label{analnonlog}
  	Let $k(t)$ be in $\mathcal{K}_a$ and $k(z)$ its complex extension. Let $\rho_0$ be the distance of the singularity of $k(z)$ nearest to the origin. Then for every positive $\epsilon<\rho_0,a$, and $\lambda\neq 1, 2, 3,\dots$,
 	\begin{equation}\label{bobbobx}
 		\mathcal{M}^*_a[k(t);1-\lambda] = \int_{\epsilon}^a \frac{k(t)}{t^{\lambda}}\,\mathrm{d}t +\sum_{l=0}^{\infty} a_l \frac{\epsilon^{l-\lambda+1}}{(l-\lambda+1)}.
 	\end{equation}
 \end{lemma}
 \begin{proof}
 	We deform the contour of integration from $C$ to $C'$. Since no pole is enclosed by the contour $C$, the deformation leads to the equality
 	\begin{equation}\label{bob}
 		\mathcal{M}^*_a[k(t);1-\lambda] = \int_{\epsilon}^a \frac{k(t)}{t^{\lambda}}\,\mathrm{d}t + \frac{1}{\left(e^{-2\pi\lambda i}-1\right)} \int_{C_{\epsilon}} \frac{k(z)}{z^{\lambda}}\,\mathrm{d} z .
 	\end{equation}
 	This time $\mathrm{Re}\lambda\geq 1$ so that the integral around $C_{\epsilon}$ does not vanish as $\epsilon\rightarrow 0$; in fact, the integral diverges in the limit so that we cannot drop it in the same way we did when we were obtaining the analytic continuation. Also the first integral diverges in the same limit. 
 	
 	We take $\epsilon$ sufficiently small so that $k(z)$ is analytic in a neighborhood of the origin that contains ${C}_{\epsilon}$. Then we can expand $k(z)$ about $z=0$, $k(z)=\sum_{l=0}^{\infty}a_l z^l$, to evaluate the integral around ${C}_{\epsilon}$. Inserting the expansion back in the integral $\int_{{C}_{\epsilon}} k(z) z^{-\lambda}\,\mathrm{d}z$ and performing term by term integration with the parametrization $z=\epsilon e^{i\theta}$, we obtain the integral around ${C}_{\epsilon}$,
 	\begin{equation}\label{bok}
 		\int_{{C}_{\epsilon}} \frac{k(z)}{z^{\lambda}}\,\mathrm{d}z = \left(e^{-2\pi\lambda i}-1\right) \sum_{l=0}^{\infty} a_l \frac{\epsilon^{l-\lambda+1}}{(l-\lambda+1)}, \;\;\; \lambda\neq 0, 1, 2, \dots . 
 	\end{equation}
 	Substituting \eqref{bok} back into \eqref{bob} yields \eqref{bobbobx}.
 \end{proof}
 
 \subsubsection{At Regular Points} 
 
 \begin{theorem}\label{hello}
 	Let $k(t)$ be in $\mathcal{K}_a$ and $k(z)$ its complex extension. Then for all non-integer $\lambda$ in the half-plane $\mathrm{Re}(\lambda)\geq1$,
 	\begin{equation}\label{qua}
 		\bbint{0}{a} \frac{k(t)}{t^{\lambda}}\,\mathrm{d}t = \mathcal{M}^*_a[k(t);1-\lambda] .
 	\end{equation}
 	Moreover, the finite-part integral assumes the contour integral representation
 	\begin{equation}\label{que}
 		\bbint{0}{a}\frac{k(t)}{t^{\lambda}}\mathrm{d}t = \frac{1}{\left(e^{-2\pi\lambda i}-1\right)} \int_{C}\frac{k(z)}{z^{\lambda}}\,\mathrm{d}z,
 	\end{equation}
 where the contour $C$ is as described in Theorem-\ref{one}. 
 \end{theorem}
 \begin{proof}
 	 	The positive number $\epsilon$ in equation-\eqref{bobbobx} is fixed, but now we let $\epsilon$ approach zero. The integral and sum in 
 	 	equation \eqref{bobbobx} separately diverge in the said limit. We group the terms of the infinite series into those that diverge as $\epsilon\rightarrow 0$ and the terms that vanish in the same limit. Let $\lambda=\lambda_R + i \lambda_I$, with $\lambda_R$ and $\lambda_I$ the real and imaginary parts of $\lambda$. The terms corresponding to $l\leq \floor{\lambda_R-1}$ diverge or indeterminate as $\epsilon \rightarrow 0$ and the rest vanish in the same limit. Then
 	\begin{equation}\label{pwe}
 		\frac{1}{\left(e^{-2\pi\lambda i}-1\right)}\int_{C_{\epsilon}} \frac{k(z)}{z^{\lambda}}\,\mathrm{d}z =  \sum_{l=0}^{\floor{\lambda_R-1}} a_l \frac{\epsilon^{l-\lambda+1}}{(l-\lambda+1)}+\sum_{l=\floor{\lambda_R-1}+1}^{\infty} a_l \frac{\epsilon^{l-\lambda+1}}{(l-\lambda+1)},  
 	\end{equation}
 	for $\lambda\neq 0, 1, 2, \dots $. Substituting this back into equation \eqref{bob} and taking the limit, the analytic continuation assumes the limit representation
 	\begin{equation}\label{keek2}
 		\mathcal{M}^*_a[k(t);1-\lambda]= \lim_{\epsilon\rightarrow 0}\left[\int_{\epsilon}^a\frac{k(t)}{t^{\lambda}}\,\mathrm{d}t - \sum_{l=0}^{\floor{\lambda_R-1}} \frac{a_l}{(\lambda-l-1) \epsilon^{\lambda-l-1}}\right],
 	\end{equation}
 	for $\mathrm{Re}\lambda\geq 1$, $\lambda\neq 1, 2, 3,\dots$ . The limit exits by virtue of the fact that the left hand side exists.
 	
 	We now show that \eqref{pwe} is just the finite-part of the divergent integral $\int_0^a k(t) t^{-\lambda}\,\mathrm{d}t$ for $\mathrm{Re}\lambda\geq 1$. Let $c>\epsilon$ and $c$ is sufficiently small such that the expansion $k(t)=\sum_{l=0}^{\infty} a_l t^l$ converges absolutely for all $t\leq c$. Then the integral can be decomposed into two parts,
 	\begin{equation}
 		\int_{\epsilon}^a\frac{k(t)}{t^{\lambda}}\,\mathrm{d}t = \int_{\epsilon}^c\frac{k(t)}{t^{\lambda}}\mathrm{d}t + \int_c^a \frac{k(t)}{t^{\lambda}}\,\mathrm{d}t.
 	\end{equation}
 	Introducing the expansion $k(t)=\sum_{l=0}^{\infty} a_l t^l$ into the first term and distributing the integration, which we can do because the each term in the series is continuous and the the series converges uniformly in the interval, we obtain 
 	\begin{eqnarray}\label{kwe}
 		\int_{\epsilon}^a\frac{k(t)}{t^{\lambda}}\,\mathrm{d}t &=&-\sum_{l=0}^{\infty} a_l \frac{\epsilon^{l-\lambda +1}}{(l-\lambda+1)}+\sum_{l=0}^{\infty} a_l \frac{c^{l-\lambda +1}}{(l-\lambda+1)} + \int_{c}^a \frac{k(t)}{t^{\lambda}}\,\mathrm{d}t .
 	\end{eqnarray}
 	Despite appearances, the right hand side of the equation is independent of $c$ as can be verified by taking its derivative with respect to $c$. It is sufficient to identify the diverging part of the integral which arises from the first term in the right hand side of \eqref{kwe}. It is given by
 	\begin{eqnarray}
 		D_{\epsilon}= \sum_{l=0}^{\floor{\lambda_R-1}}\frac{a_l}{(\lambda-l-1) \epsilon^{\lambda-l-1}}.
 	\end{eqnarray}
 	Comparing this with the left hand side of equation \eqref{keek2}, we find that the entire expression is just the finite-part of the divergent integral. Then we have established equation \eqref{qua}; and, from this, the contour integral representation of the finite-part integral is equal to the analytic continuation itself as given by equation \eqref{que}.
 \end{proof}
 
 \subsubsection{At Isolated Singularities}
 \begin{theorem}
 	Let $k(t)$ be in $\mathcal{K}_a$ and $k(z)$ its complex extension. Then for all positive integer $m$, 
 	\begin{equation}\label{poleLim}
 		\bbint{0}{a}\frac{k(t)}{t^m}\mathrm{d}t = \reglim{\lambda}{m} \mathcal{M}^*_a[k(t);1-\lambda].
 	\end{equation}
 	Moreover, the finite-part integral assumes the contour integral representation
 	\begin{equation}\label{poleIso}
 		\bbint{0}{a} \frac{k(t)}{t^m}\,\mathrm{d}t = \frac{1}{2\pi i} \int_C \frac{k(z)}{z^m}\left(\log z - i\pi\right)\,\mathrm{d} z ,
 	\end{equation}
 where the contour $C$ is as described in Theorem-\ref{one}. 
 \end{theorem}
 \begin{proof} We divide the proof into removable and simple pole singularities. In both cases, we rationalize the analytic continuation in the form
 	\begin{equation}
 		\mathcal{M}^*_a[k(t);1-\lambda]= \frac{f(\lambda)}{g(\lambda)},
 	\end{equation}
 	where
 	\begin{equation}
 		f(\lambda)= \int_C \frac{k(t)}{z^{\lambda}}\,\mathrm{d}z, \;\;\; g(\lambda) = \left(e^{-2\pi \lambda i}-1\right) .
 	\end{equation}
 	Both $f(\lambda)$ and $g(\lambda)$ are analytic everywhere in the strip of analyticity of the analytic continuation of the Mellin transform.
 	 	
 	\subsubsection*{At simple poles} 
 	For $k^{(m-1)}(0)\neq 0$ for positive integer $m$, the point $\lambda=m$ is a simple pole of the analytic continuation. Then the regularized limit at $\lambda=m$ is given by equation \eqref{reglimsimple} where $\lambda_0=m$, in particular
 	\begin{equation}\label{rv0}
 		\reglim{\lambda}{m}\frac{f(\lambda)}{g(\lambda)} = \frac{f'(m)}{g'(m)}- f(m)\frac{g''(m)}{2 (g'(m))^2} .
 	\end{equation}
 	Performing the indicated differentiations and evaluating the resulting expressions at $\lambda=m$ yield the regularized limit at the poles,
 	\begin{equation}\label{polerep}
 		\reglim{\lambda}{m}\mathcal{M}^*_a[k(t);1-\lambda] 
 		= \frac{1}{2\pi i} \int_C \frac{k(z)}{z^m}\left(\log z - i\pi\right)\,\mathrm{d} z .
 	\end{equation}
 	
 	Using the same method employed to establish Theorem-\ref{hello}, we can also establish that the regularized limit equals the finite part of the divergent integral $\int_0^a k(t) t^{-m},\,\mathrm{d}t$. We deform the contour of integration ${C}$ to ${C}'$ as indicated in Figure-\ref{boat1} in \eqref{polerep} and evaluate the integral around ${C}_{\epsilon}$ under the same conditions leading to equation \eqref{bok}. As in equation \eqref{pwe}, the integral around $C_{\epsilon}$ has diverging and converging parts as $\epsilon\rightarrow 0$. Taking the limit as $\epsilon\rightarrow 0$, we obtain the limit representation of the regularized limit at the poles,
 	\begin{equation}\label{kikiki}
 		\reglim{\lambda}{m}\mathcal{M}^*_a[k(t);1-\lambda] =  \lim_{\epsilon\rightarrow 0^+}\left[\int_{\epsilon}^a \frac{k(t)}{t^{m}}\,\mathrm{d}t  -\sum_{k=0}^{n-1}\frac{k^{(k)}(0)}{k! (m-2)}\frac{1}{\epsilon^{n-k}} + \frac{k^{(m-1)}(0)}{n!}\ln \epsilon\right].
 	\end{equation}
 	By extracting the diverging and converging parts of $\int_{\epsilon}^a k(t) t^{-m}\,\mathrm{d}t$ in the same we have done in Theorem-\ref{hello}, we find that the right hand side of \eqref{kikiki} is just the finite part of the divergent integral. Thus we have established that equations \eqref{poleLim} and \eqref{poleIso} hold at the simple poles of the analytic continuation.
 
 \subsubsection*{At removable singularities}
 Let $k^{(m-1)}(0)=0$ for some positive integer $m$ so that $\lambda=m$ is a removable singularity of the analytic continuation $\mathcal{M}^*_a[k(t);1-\lambda]$. Then the regularized limit coincides with the Cauchy limit
 \begin{equation}\label{cauchy}
 	\lim_{\lambda\rightarrow m}\mathcal{M}^*_a[k(t);1-\lambda]= \frac{f'(m)}{g'(m)}. 
 \end{equation}
 Performing the indicated differentiations at $\lambda=m$, we obtain 
 \begin{equation}\label{reglim_removable}
 	\lim_{\lambda\rightarrow m}\mathcal{M}^*_a[k(t);1-\lambda] = \frac{1}{2\pi i} \int_C \frac{k(z)}{z^m} \log z\,\mathrm{d} z .
 \end{equation}
 Following the same method as above, we can establish that \eqref{reglim_removable} is the finite-part integral when $\lambda=m$ is a removable singularity.
 
 Observe that the contour integral in \eqref{reglim_removable} is the reduction of the contour integral in \eqref{poleIso} for $\int_C k(z) z^{-m}\,\mathrm{d}z =0$ or for $k^{(m-1)}(0)=0$. Since the regularized limit coincides with the Cauchy limit for the case of removable singularities, we have thus shown that equations \eqref{poleLim} and \eqref{poleIso} hold when $k^{(m-1)}(0)=0$ or when $\lambda=m$ is a removable singularity of the analytic continuation of the Mellin transform. 
\end{proof}

Our result \eqref{poleLim} generalizes our earlier result in \cite{galapon1} where we had the restriction $\lambda=m+\nu$ for positive integer $m$ and $0<\nu<1$. Moreover, there the contour integral representation for the finite-part integral corresponding to $\nu=0$ was introduced independent of the case for non-integer $\lambda$. Now we see that both contour integral representations come from the same analytic continuation of the Mellin transform.
 
 \subsection{Example} Let us demonstrate how finite-part integrals can be conveniently extracted using tabulated Mellin transforms. Let us evaluate the following finite part integral using the method of analytic continuation,
 \begin{equation}
 	\bbint{0}{\infty} \frac{\cos(a t)}{t^{\lambda}}\,\mathrm{d}t ,\;\;\; \mathrm{Re}(\lambda)\geq 1 .
 \end{equation}
 We identify $k(t)=\cos(a t)$ which belongs to $\mathcal{K}_{\infty}$. Since $k(t)$ is even in $t$, $k^{(2n-1)}(0)=0$ for $\lambda=2n$, $n=1, 2,\dots$,  so that the points $\lambda=2n$ are removable singularities. On the other hand, $k^{(2n-2)}(0)\neq 0$ for the points $(2n-1)$, so that the points $\lambda=(2n-1)$ are simple poles of the analytic continuation of the Mellin transform. From the tabulated integral in \cite[p.441,\#3.761.9]{gr}, we deduce the integral
 \begin{equation}\label{popo}
 	\int_0^{\infty} \frac{\cos(a t)}{t^{\lambda}}\,\mathrm{d}t =\frac{\pi a^{\lambda-1} \sin(\pi\lambda/2)}{\sin(\pi\lambda) \Gamma(\lambda)}, \;\;\; 0<\mathrm{Re}(\lambda)<1 ,
 \end{equation}
 which is the desired Mellin transform $\mathcal{M}[\cos(a t); 1-\lambda]$. The right hand side of \eqref{popo} analytically continues the Mellin transform in the entire complex plane, 
 \begin{equation}
 	\mathcal{M}^*[\cos(a t);1-\lambda]=\frac{\pi a^{\lambda-1} \sin(\pi\lambda/2)}{\sin(\pi\lambda) \Gamma(\lambda)}.
 \end{equation}   
 Clearly the points $\lambda=2n$  are removable singularities and the points $\lambda=2n-1$ are simple poles in the half-plane $\mathrm{Re}(\lambda)\geq 1$, and all other points are regular points.
 
 The finite-part integrals can now be calculated from the analytic continuation of the Mellin transform. At the regular points, the finite-part equals the analytic continuation so that 
 \begin{equation}
 	\begin{split}\label{regular}
 		\bbint{0}{\infty}\frac{\cos(a t)}{t^{\lambda}}\,\mathrm{d}t = \frac{\pi a^{\lambda-1} \sin(\pi\lambda/2)}{\sin(\pi\lambda) \Gamma(\lambda)}, \;\; \mathrm{Re}(\lambda)\geq 1, \;\; \lambda\neq 1, 2, 3,\dots .
 	\end{split}
 \end{equation}
 To evaluate the finite-parts at the isolated singularities of $\mathcal{M}^*[\cos(at); 1-\lambda]$, we use the rationalization of the analytic continuation, \begin{equation}
 	\mathcal{M}^*[\cos(a );1-\lambda]=\frac{f(\lambda)}{g(\lambda)},
 \end{equation}
 where 
 \begin{equation}
 	f(\lambda)=\pi a^{\lambda-1} \frac{\sin(\pi\lambda/2)}{\Gamma(\lambda)}, \;\;\; g(\lambda)=\sin(\pi\lambda).
 \end{equation} 
 At the removable singularities, $\lambda=2n$, the finite-part is equal to the usual Cauchy limit of $\mathcal{M}^*[\cos(a t);1-\lambda]$ as $\lambda\rightarrow(2n-1)$,
 \begin{equation}\label{removable}
 	\bbint{0}{\infty}\frac{\cos(a t)}{t^{2n}}\,\mathrm{d}t=\frac{(-1)^n \pi 2^{2n-1} }{2 (2n-1)!},\;\;\; n=1, 2, 3,\dots  .
 \end{equation}
 On the other hand, at the simple poles, $\lambda=2n-1$, the finite-part is equal to the regularized limit of $\mathcal{M}^*[\cos(a t);1-\lambda]$ as $\lambda\rightarrow(2n-1)$, 
 \begin{equation}\label{pole}
 	\bbint{0}{\infty}\frac{\cos(a t)}{t^{2n-1}}\,\mathrm{d}t= \frac{(-1)^{n+1} a^{2n-2}}{(2n-n)!} \left[\ln a -\psi(2n-1)\right],\;\;\; n=1, 2, 3, \dots .
 \end{equation} 
 Alternatively the finite-part integrals \eqref{regular}, \eqref{removable} and \eqref{pole}  can be obtained using the canonical definition \eqref{fpidef} and \eqref{fpidef2} of the finite-part integral only with much more effort.

\section{The Relationship Between Analytic Continuation and Finite-part Integral in the Presence of Logarithmic Singularities}\label{powerlog}
Given $k(t)$ an element of $\mathcal{K}_a$, we now wish to obtain the relationship between the finite-part integral
\begin{equation}
	\bbint{0}{a} \frac{k(t)\ln^n t}{t^{\lambda}}\,\mathrm{d}t, \mathrm{Re}(\lambda)\geq 1,  
\end{equation}
and the Mellin transform
\begin{equation}\label{meme}
	\mathcal{M}_a[k(t) \ln^n t;1-\lambda]=\int_0^a \frac{k(t)\ln^n t}{t^{\lambda}}\,\mathrm{d}t, \;\;\; d<\mathrm{Re}(\lambda)<1,\;\;\; 0<a\leq \infty,  
\end{equation}
for all positive integer $n$. We first establish the existence of the Mellin transform \eqref{meme} given the conditions on $k(t)$.  
\begin{theorem}\label{co}
	Let $k(t)$ be in $\mathcal{K}_a$. Then the Mellin transform $\mathcal{M}_a[k(t)\ln^n t; 1-\lambda]$ for all positive integer $n$ exists and its strip of analyticity coincides with the strip of analyticity of the Mellin transform $\mathcal{M}_a[k(t);1-\lambda]$.
\end{theorem}
\begin{proof}
First let us consider the case of finite $a$. Under this condition the strip of analyticity of the Mellin transform $\mathcal{M}[k(t);1-\lambda]$ is the half plane $\mathrm{Re}(\lambda)<1$. Since $t^{-\lambda} k(t)\ln^n t$ is locally integrable in the neighborhood of $t=a$, the strip of analyticity of $\int_0^a t^{-\lambda} k(t) \ln^n t\,\mathrm{d}t$ still goes all the way in the negative real axis. Near the origin $\int t^{-\lambda} k(t) \ln^n t\, \mathrm{d}t =O(t^{1-\lambda}\ln^n t)$ which vanishes for all $n$ as $t\rightarrow 0^+$ provided $\mathrm{Re}(\lambda)<1$, owing to the slow growth of $\ln^n t$ compared to the decay of $t^{1-\lambda}$ as $t\rightarrow 0$. Then $\int_0^a t^{-\lambda} k(t) \ln^nt\, \mathrm{d}t$ exists in the strip $\mathrm{Re}(\lambda)<1$, which is the strip of analyticity of the $\mathcal{M}_a[k(t);1-\lambda]$. 

For this case of $a=\infty$, the strip of analyticity is $d<\mathrm{Re}(\lambda)<1$ for some $d<1$. The existence of $\int_0^{\infty} t^{-\lambda} k(t) \mathrm{d}t$ implies that $t^{-\lambda} k(t)$ is integrable at infinity. Now the rate of increase of $\ln^n t$ with $t$ vanishes as $t\rightarrow\infty$; this means that the behavior of $t^{-\lambda} k(t) \ln^n t$ at infinity is dominated by the behavior of  $t^{-\lambda} k(t)$ there. That is if $t^{-\lambda} k(t)$ is integrable at infinity, $t^{-\lambda} k(t)\ln^n t$ is likewise integrable there for any positive integer $n$. Thus a necessary condition for $\int_0^{\infty} t^{-\lambda} k(t) \ln^nt\, \mathrm{d}t$ to exist is that $d<\mathrm{Re}(\lambda)$. Now from the same reasoning given above, $t^{-\lambda} k(t)\ln^n t$ is integrable at the origin provided $\mathrm{Re}(\lambda)<1$. Thus $\int_0^{\infty} t^{-\lambda} k(t) \ln^n \mathrm{d}t$ exists in the strip $d<\mathrm{Re}(\lambda)<1$, which is the strip of analyticity of $\mathcal{M}_a[k(t);1-\lambda]$.
\end{proof}

\subsection{Analytic Continuation}
\begin{theorem}\label{ako}
	Let $k(t)$ be in $\mathcal{K}_a$. Then for every positive integer $n$
	\begin{equation}\label{generalanal}
		\mathcal{M}_a^*[k(t)\ln^nt;1-\lambda] = (-1)^n \frac{\mathrm{d}^n}{\mathrm{d}\lambda^n} \mathcal{M}_a^*[k(t);1-\lambda]. 
	\end{equation}
	If $k^{(m-1)}(0)\neq 0$, then $\lambda=m$ is a pole of $\mathcal{M}_a^*[k(t)\ln^nt;1-\lambda]$ of order $(n+1)$; otherwise, when $k^{(m-1)}(0)=0$, $\lambda=m$ is a removable singularity.  
\end{theorem}
\begin{proof} 
We employ Lemma-\ref{analnonlog}, in particular equation \eqref{bobbobx}. The integral in the right hand side can be differentiated under the integral sign to any order and the infinite series can be differentiated term by term to any order. Then
\begin{equation}\label{xoxo}
	\begin{split}
		(-1)^n \frac{\mathrm{d}^n}{\mathrm{d}\lambda^n} \mathcal{M}_a^*[k(t);1-\lambda] = & \int_{\epsilon}^{a} \frac{k(t) \ln^n t}{t^{\lambda}}\,\mathrm{d}t \\
		&\hspace{-12mm}+ (-1)^n n! \sum_{l=0}^{\infty} a_l \epsilon^{l-\lambda+1}\sum_{j=0}^n \frac{(-1)^j}{j!} \frac{\ln^j\epsilon}{(l-\lambda+1)^{n-j+1}}
	\end{split}
\end{equation}
When $\mathrm{Re}(\lambda)<1$ the second term in equation \eqref{xoxo} vanishes as $\epsilon\rightarrow 0$ and the first term is just the integral $\int_0^a t^{-\lambda} k(t) \ln^n t\,\mathrm{d}t$ in the same limit. According to Theorem-\ref{co} this integral exists and has a strip of analyticity coinciding with that of  $\mathcal{M}_a[k(t);1-\lambda]$. Thus
\begin{equation}\label{de}
	\begin{split}
		(-1)^n \frac{\mathrm{d}^n}{\mathrm{d}\lambda^n} \mathcal{M}_a^*[k(t);1-\lambda] =  \mathcal{M}_a[k(t)\ln^n t;1-\lambda],\;\; \mathrm{Re}(\lambda)<1 . 
	\end{split}
\end{equation}
On the other hand, for some fixed positive $\epsilon<\rho_0$, the right hand side of equation \eqref{xoxo} is defined everywhere in the region $\mathrm{Re}(\lambda)\geq 1$ except at some isolated points, and, thus, analytically continues the right hand side of \eqref{de} in the region $\mathrm{Re}(\lambda)\geq 1$. This establishes the equality \eqref{generalanal}. 
	
Now the condition $k^{(m-1)}(0)\neq 0$ implies that the coefficient $a_{m-1}$ in the expansion of $k(t)$ is not equal to zero. Then from the second term of \eqref{xoxo}, it is seen that the term $l=(m-1)$ (and only this term) develops a pole at $\lambda=m-1$ of order $(n+1)$; thus, $\mathcal{M}_a[k(t)\ln^nt;1-\lambda]$ has a pole of order $(n+1)$ there. On the other hand, when $k^{(m-1)}(0)=0$, the coefficient $a_{m-1}$ is necessarily zero. Then the $l=(m-1)$ term is not present and the pole singularity is removed; thus $\mathcal{M}_a[k(t)\ln^nt;1-\lambda]$ has at most a removable singularity there. 
\end{proof}

\subsection{The Relationship Between the Analytic Continuation and Finite-part Integrals} 

\subsection{At Regular Points}
\begin{theorem}\label{bay} 
	Let $k(t)$ be in $\mathcal{K}_a$. If $\lambda\neq 1, 2, ,3,\dots$, then 
	\begin{equation}\label{hehe}
		\bbint{0}{a}\frac{k(t) \ln^nt}{t^{\lambda}}\,\mathrm{d}t = \mathcal{M}^*[k(t) \ln^nt;1-\lambda],\;\; \mathrm{\lambda}>1 ,
	\end{equation}
	for all positive integer $n$. 
\end{theorem}
\begin{proof} 
	Each term in the infinite series in the representation of $\mathcal{M}^*[k(t)\ln^n t;1-\lambda]$ given by equation \eqref{xoxo} either diverges or vanishes as $\epsilon$ approaches zero when $\mathrm{Re}(\lambda)>1$. The integral in the right hand side diverges in the same limit. However, the left hand side is fixed so that the limit of the right hand side exists as $\epsilon\rightarrow 0$. The terms rendering the series divergent must necessarily cancel the divergence coming from the integral. Then
	\begin{equation}\label{he}
		\begin{split}
			\mathcal{M}_a^*[k(t) \ln^n t; 1-\lambda] = & \lim_{\epsilon\rightarrow 0}\left[ \int_{\epsilon}^{a} \frac{k(t) \ln^n t}{t^{\lambda}}\,\mathrm{d}t\right. \\
			&\hspace{-20mm}\left.+ (-1)^n n! \sum_{l=0}^{\floor{\mathrm{Re}(\lambda)-1}} a_l \epsilon^{l-\lambda+1}\sum_{j=0}^n \frac{(-1)^j}{j!} \frac{\ln^j\epsilon}{(l-\lambda+1)^{n-j+1}}\right]
		\end{split}
	\end{equation}
	
	The desired equality \eqref{hehe} is established if the second term of \eqref{he} happens to be the negative of the divergent part of the integral $\int_{\epsilon}^{a} t^{-\lambda}\,k(t) \ln^nt \, \mathrm{d}t$ as $\epsilon\rightarrow 0$. Let $c$ be such that $0<\epsilon<c<a$. Then we have the decomposition
	\begin{equation}\label{decom}
		\int_{\epsilon}^{a} \frac{k(t)\ln^n t}{t^{\lambda}}\,\mathrm{d}t =\int_{\epsilon}^{c} \frac{k(t)\ln^n t}{t^{\lambda}}\,\mathrm{d}t + \int_{c}^{a} \frac{k(t)\ln^n t}{t^{\lambda}}\,\mathrm{d}t .
	\end{equation}
	  Let us work through the first term. Let $c$ be sufficiently small such that the expansion $k(t)=\sum_{l=0}^{\infty}a_l t^l$ holds in the interval $[0,c+\delta]$ for some small $\delta>0$. We can then expand $k(t)$ about $t=0$ and perform a term by term integration, 
	  \begin{equation}\label{sese}
	  	\int_{\epsilon}^c \frac{k(t)\ln^nt}{t^{\lambda}}\,\mathrm{d}t = \sum_{l=0}^{\infty} a_l \int_{\epsilon}^c t^{l-\lambda} \ln^n t\, \mathrm{d}t .
	  \end{equation}
	   Employing the known integral \cite[p-238, \#2.722]{gr},
	\begin{equation}\label{integral}
		\int x^s \ln^r x\,\mathrm{d}x = (-1)^r r! \,x^{s+1} \sum_{j=0}^r \frac{(-1)^j}{j!} \frac{\ln^j x}{(s+1)^{r-j+1}},
	\end{equation}
	in evaluating the integrals in equation \eqref{sese} and returning the result back into \eqref{decom}, equation \eqref{decom} assumes the form
	\begin{equation}\label{le}
		\begin{split}
			\int_{\epsilon}^{a} \frac{k(t)\ln^n t}{t^{\lambda}}\,\mathrm{d}t =& \int_{c}^{a} \frac{k(t)\ln^n t}{t^{\lambda}}\,\mathrm{d}t \\
			&+ (-1)^n n!\sum_{l=0}^{\infty} a_l c^{l-\lambda+1} \sum_{j=0}^n \frac{(-1)^n}{j!} \frac{\ln^j c}{(l-\lambda+1)^{n-j+1}}\\
			&-(-1)^n n!\sum_{l=0}^{\infty} a_l \epsilon^{l-\lambda+1} \sum_{j=0}^n \frac{(-1)^n}{j!} \frac{\ln^j \epsilon}{(l-\lambda+1)^{n-j+1}}.
		\end{split}
	\end{equation}
	
	Each term in the third term of \eqref{le} either diverge or vanish. The vanishing terms and the first two terms constitute the converging part of the left hand side of the equation \eqref{le}. The terms in the range $0\leq l<\mathrm{Re}(\lambda)-1$ of the third term diverge or indeterminate as $\epsilon$ approaches zero. Thus we have identified the divergent part of the left hand side and is given by
	\begin{equation}
		D_{\epsilon}=-(-1)^n n! \sum_{l=0}^{\floor{\mathrm{Re}(\lambda)-1}} a_l \epsilon^{l-\lambda+1}\sum_{j=0}^n \frac{(-1)^j}{j!} \frac{\ln^j\epsilon}{(l-\lambda+1)^{n-j+1}} .
	\end{equation} Comparing this with equation \eqref{he}, we verify that the second term in the right hand side of \eqref{he} is the negative of the divergent part of the integral $\int_{\epsilon}^a t^{-\lambda}k(t)\ln^n t\,\mathrm{d}t$ as $\epsilon$ becomes arbitrarily small. 	
\end{proof}

From equality \eqref{xoxo} and from Theorem-\ref{bay} itself, we obtain the following representation of the finite-part integral.
\begin{corollary} \label{colloanal}
Let $k(t)$ be in $\mathcal{K}_a$ and $k(z)$ its complex extension.	Let $\rho_0$ be the distance of the singularity of $k(z)$ nearest to the origin. Then for every positive $\epsilon<a,\rho_0$, and for all $\lambda\neq 1, 2, 3,\dots$,
	 	\begin{equation}\label{xoxo2}
		\begin{split}
			\bbint{0}{a} \frac{k(t) \ln^n t}{t^{\lambda}}  = & \int_{\epsilon}^{a} \frac{k(t) \ln^n t}{t^{\lambda}}\,\mathrm{d}t \\
			&\hspace{-12mm}+ (-1)^n n! \sum_{l=0}^{\infty} a_l \epsilon^{l-\lambda+1}\sum_{j=0}^n \frac{(-1)^j}{j!} \frac{\ln^j\epsilon}{(l-\lambda+1)^{n-j+1}} ,\;\; \mathrm{Re}(\lambda)>1,
		\end{split}
	\end{equation}
for all positive integer $n$. 
\end{corollary}
\subsection{At Isolated Singularities}
\begin{theorem}\label{bek}
	Let $k(t)$ be in $\mathcal{K}_a$. Then for all positive integers $n$ and $m$,
	\begin{equation}\label{key}
		\bbint{0}{a}\frac{k(t) \ln^nt}{t^{m}}\,\mathrm{d}t = \reglim{\lambda}{m}\mathcal{M}_a^*[k(t)\ln^n t;1-\lambda] .
	\end{equation}
\end{theorem}

\begin{proof}The analytic continuation \eqref{xoxo} is now divergent at $\lambda=m$ for a positive integer $m$. The divergence occurs at $l=m-1$. We isolate this term. We then use the linearity of the regularized limit to arrive at
	\begin{equation}
		\begin{split}
			\reglim{\lambda}{m}\mathcal{M}_a^*[k(t)\ln^n t;1-\lambda]=&\reglim{\lambda}{m}\int_{\epsilon}^{a} \frac{k(t) \ln^n t}{t^{\lambda}}\,\mathrm{d}t \\
			&\hspace{-28mm}+  \sum_{j=0}^n \frac{(-1)^j}{j!}\ln^j\epsilon  \left[\sum_{l=0}^{m-2} a_l \reglim{\lambda}{m} \frac{\epsilon^{l-\lambda+1}}{(l-\lambda+1)^{n-j+1}}
			\right.\\
			&\hspace{-28mm} 
			\left.+ a_{m-1} \reglim{\lambda}{m} \frac{\epsilon^{m-\lambda}}{(m-\lambda)^{n-j+1}}\right.\\ &\hspace{-28mm}\left.+\sum_{j=0}^n \frac{(-1)^j}{j!}\ln^j\epsilon  \sum_{l=m}^{\infty} a_l \reglim{\lambda}{m} \frac{\epsilon^{l-\lambda+1}}{(l-\lambda+1)^{n-j+1}}\right] .
		\end{split} 
	\end{equation}
	All the regularized limit, except the one in the $a_{m-1}$ term, reduce to the Cauchy limit. We appeal to the result \eqref{coco} to obtain
	\begin{equation}
		\reglim{\lambda}{m} \frac{\epsilon^{m-\lambda}}{(m-\lambda)^{n-j+1}}=\frac{\ln^{n-j+1}\epsilon}{(n-j+1)!} .
	\end{equation}
	Then
	\begin{equation}\label{labo}
		\begin{split}
			\reglim{\lambda}{m}\mathcal{M}_a^*[k(t)(\ln t)^n;1-\lambda]=&\int_{\epsilon}^{a} \frac{k(t) \ln^n t}{t^{m}}\,\mathrm{d}t \\
			&\hspace{-24mm}+  \sum_{j=0}^n \frac{(-1)^j}{j!}\ln^j\epsilon  \sum_{l=0}^{m-2} a_l  \frac{\epsilon^{l-m+1}}{(l-m+1)^{n-j+1}}
			+ a_{m-1} \frac{\ln^{n+1}\epsilon}{(n+1)}\\ &\hspace{-24mm}+\sum_{j=0}^n \frac{(-1)^j}{j!}\ln^j\epsilon  \sum_{l=m}^{\infty} a_l  \frac{\epsilon^{l-m+1}}{(l-m+1)^{n-j+1}},
		\end{split} 
	\end{equation}
	where the coefficient of the $a_{m-1}$ term has been simplified using the fact that
	\begin{equation}
		\sum_{j=0}^n \frac{(-1)^j}{j! (n-j+1)!} = \frac{(-1)^n}{(n+1)!} .
	\end{equation}
	
	The first three terms diverge as $\epsilon$ approaches zero and the last term vanishes in the same limit. The left hand side is fixed so that the limit must exist. Then
	\begin{equation}
		\begin{split}
			\reglim{\lambda}{m}\mathcal{M}_a^*[k(t)\ln^n t;1-\lambda]=&\lim_{\epsilon\rightarrow 0}\left[\int_{\epsilon}^{a} \frac{k(t) \ln^n t}{t^{m}}\,\mathrm{d}t\right. \\
			&\hspace{-30mm}\left.+  \sum_{j=0}^n \frac{(-1)^j}{j!}\ln^j\epsilon  \sum_{l=0}^{m-2} a_l  \frac{\epsilon^{l-m+1}}{(l-m+1)^{n-j+1}}
			+ a_{m-1} \frac{\ln^{n+1}\epsilon}{(n+1)}\right],
		\end{split} 
	\end{equation}
	Proceeding in the same manner we have proceeded in Theorem-\ref{bay} and using, in conjunction with \eqref{integral}, the integral
	\begin{equation}
		\int \frac{\ln^r x}{x}\,\mathrm{d}x = \frac{\ln^{r+1}x}{r+1},
	\end{equation}
	we can establish that the second term is the negative of the divergent part of the integral $\int_{\epsilon}^{a} k(t) (\ln t)^n t^{-m}\,\mathrm{d}t$. This establishes equation \eqref{key}. 
\end{proof}

From equation \eqref{labo} and from Theorem-\ref{bek} itself, we obtain the following representation for the finite-part integral.
\begin{corollary}\label{collozero} 
Under the same relevant conditions as in Corollary-\ref{colloanal},
	\begin{equation}\label{xoxo3}
		\begin{split}
			\bbint{0}{a} \frac{k(t) \ln^n t}{t^{m}}  = & \int_{\epsilon}^{a} \frac{k(t) \ln^n t}{t^{m}}\,\mathrm{d}t \\
			&\hspace{-15mm}+  \sum_{j=0}^n \frac{(-1)^j}{j!}\ln^j\epsilon  \sum_{l=0}^{m-2} a_l  \frac{\epsilon^{l-m+1}}{(l-m+1)^{n-j+1}}
			+ a_{m-1} \frac{\ln^{n+1}\epsilon}{(n+1)}\\ &\hspace{-15mm}+\sum_{j=0}^n \frac{(-1)^j}{j!}\ln^j\epsilon  \sum_{l=m}^{\infty} a_l  \frac{\epsilon^{l-m+1}}{(l-m+1)^{n-j+1}} ,
		\end{split}
	\end{equation}
	for all positive integer $m$. 
\end{corollary}

\begin{figure}
	\includegraphics[scale=0.5]{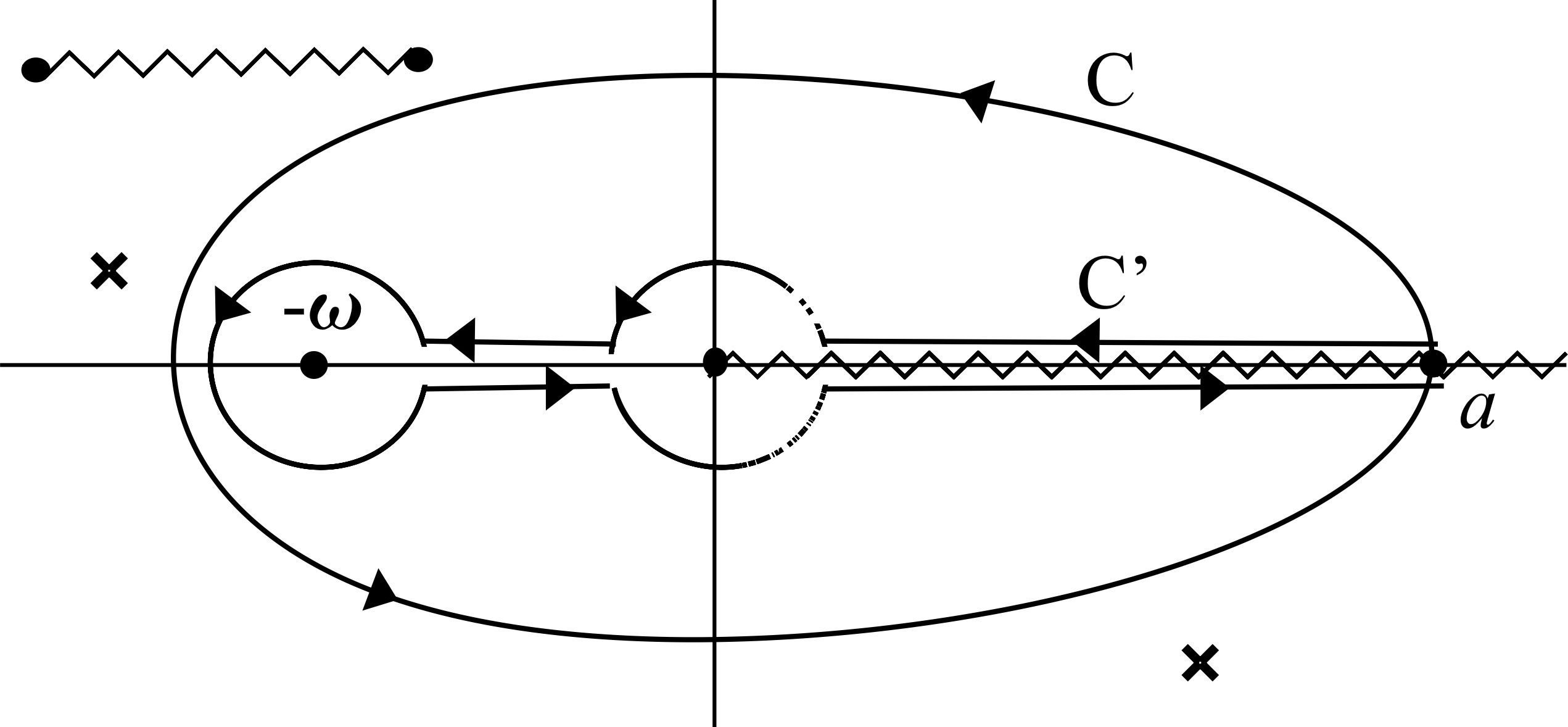}
	\caption{The path of integration  in performing finite-part integration of the Stieltjes transform. All singularities of $k(z)$ must stay to right of $\mathrm{C}$ and $-\omega$ must stay to the left when $\mathrm{C}$ is traversed in the counter-clockwise direction.}
	\label{fig:boat2}
\end{figure}

\subsection{Complete Characterization of the Finite-part Integral}
If we apply the definition of the finite-part integral in the strip of analyticity of $\mathcal{M}_a[k(t)\ln^n t;1-\lambda]$ for all non-negative integer $n$, we have
\begin{equation}\label{fpi1extend}
	\begin{split}
		\bbint{0}{a}\frac{k(t)\ln^n t}{t^{\lambda}}\,\mathrm{d}t = \lim_{\epsilon\rightarrow 0}\left(\int_{\epsilon}^a \frac{k(t) \ln^n t}{t^{\lambda}}\,\mathrm{d}t - D_{\epsilon}\right),\;\; d<\mathrm{Re}(\lambda)<1.
	\end{split}
\end{equation}
By hypothesis $k(t)$ belongs to $\mathcal{K}_a$ so that the integral 
\begin{equation}
	\int_0^a \frac{ k(t)\ln^n t}{t^{\lambda}}\,\mathrm{d}t=\lim_{\epsilon\rightarrow 0}\int_{\epsilon}^a \frac{ k(t)\ln^n t}{t^{\lambda}}\,\mathrm{d}t
\end{equation}
exists in the strip $d<\mathrm{Re}(\lambda)<1$. This implies that the divergent part $D_{\epsilon}$ is identically zero. Then we have the equality
\begin{equation}
	\bbint{0}{a}\frac{k(t)\ln^n t}{t^{\lambda}}\,\mathrm{d}t=\mathcal{M}_a[k(t)\ln^n t;1-\lambda],\;\; d<\mathrm{Re}(\lambda)<1 ,
\end{equation}
for all non-negative integer $n$. That is the finite-part is equal to the value of the Mellin integral when it happens to be convergent. 

This means that we can extend the domain of the finite-part integral in the entire strip of analyticity of the analytic continuation of the Mellin transform, which is the half-plane $d<\mathrm{Re}(\lambda)$ when $d$ happens to be finite or the entire complex plane when $d$ is negative infinite. Then as a function, the finite-part integral is given by
\begin{equation}\label{mok}
	\bbint{0}{a}\frac{k(t) \ln^n t}{t^{\lambda}}\,\mathrm{d}t = \left\{
	\begin{array}{ll}
		\mathcal{M}_a[k(t)\ln^n t; 1-\lambda] &,\; d<\mathrm{Re}(\lambda)<1 \\
		\mathcal{M}_a^*[k(t)\ln^n t; 1-\lambda] &,\; 1\leq \mathrm{Re}(\lambda),\;\; \lambda\neq 1, 2, 3, \dots 
		\\
		\reglim{\lambda}{m} \mathcal{M}_a^*[k(t)\ln^n t; 1-\lambda] &,\; \lambda = 1, 2, 3, \dots 
	\end{array}
	\right. 
\end{equation}
for all non-negative integer $n$. Clearly the finite-part integral as a function of $\lambda$ is almost everywhere equal to the analytic continuation of the Mellin transform. They only differ in a set of measure zero, at the isolated points where the analytic continuation develops poles. Because of the relationship between the divergent integral \eqref{div} and the Mellin transform \eqref{mt} given by \eqref{mok}, we  refer to the divergent integral \eqref{div} as Mellin-type divergent integral. 

Since the domain of the analytic continuation extends to the domain of the Mellin transform, we have established the following result.
\begin{theorem}
	The finite-part of Mellin-type divergent integral is the regularized analytic continuation of the corresponding Mellin transform, in particular,
	\begin{equation}
	\bbint{0}{a}\frac{k(t) \ln^n t}{t^{\lambda}}\,\mathrm{d}t =	\overset{\times}{\mathcal{M}_a^*}[k(t)\ln^n t; 1-\lambda],\;\; d<\operatorname{Re}(\lambda) ,
	\end{equation}
for all non-negative integer $n$.
\end{theorem}


\subsection{Contour integral representation of finite-part integral for linear logarithmic case}
We apply our results in deriving the contour integral representation of the finite-part integral for the linear logarithmic case to allow us to perform direct finite-part integration later on the corresponding Stieltjes transform. 
\begin{theorem}
	Let $k(t)$ be in $\mathcal{K}_a$ and $k(z)$ its complex extension. Then for all non-positive integer $\lambda$ with $\mathrm{Re}(\lambda)\leq 1$, 	\begin{equation}\label{n1regular}
		\begin{split}
			\bbint{0}{a}\frac{k(t) \ln t}{t^{\lambda}}\mathrm{d}t = \int_{C}\frac{k(z) }{z^{\lambda}}\, \left[\frac{\log z}{\left(e^{-2\pi\lambda i}-1\right)}  -\frac{2\pi i e^{-2\pi\lambda i}}{(e^{-2\pi\lambda i}-1)^2}\right]\mathrm{d}z,
		\end{split}
	\end{equation}
	and for all positive integer $m$,
	\begin{equation}\label{fpi2c}
		\begin{split}
			\bbint{0}{a} \frac{k(t) \ln t}{t^m}\,\mathrm{d}t =  \frac{1}{2\pi i} \int_C \frac{k(t)}{z^m}\left[\frac{1}{2}\log^2 z - \pi i \log z - \frac{\pi^2}{3}\right]\,\mathrm{d}z ,
		\end{split}
	\end{equation}
where the contour $C$ is as described in Theorem-\ref{one}.
\end{theorem}
\begin{proof}
	The analytic continuation of the Mellin transform for $n=1$ can be obtained in at least two ways. One is by proceeding in the same manner that we have arrived at the analytic continuation for $n=0$ by considering the contour integral $\int_{C} z^{-\lambda}k(z) \log z\,\mathrm{d}z$ and deforming the contour to the keyhole contour under the condition $\mathrm{Re}(\lambda)<1$. Second is by differentiation of the analytic continuation for the $n=0$ case in accordance with Theorem-\ref{ako}. Both yield the analytic continuation
	\begin{equation}\label{helo}
		\begin{split}
			\mathcal{M}_a^*[k(t)\ln t;1-\lambda] = \int_{C}\frac{k(z) }{z^{\lambda}}\, \left[\frac{\log z}{\left(e^{-2\pi\lambda i}-1\right)}  -\frac{2\pi i e^{-2\pi\lambda i}}{(e^{-2\pi\lambda i}-1)^2}\right]\mathrm{d}z.
		\end{split}
	\end{equation} 
	This is analytic everywhere in the strip $\mathrm{Re}(\lambda)\geq 1$ except possibly at positive integers $\lambda=m=1, 2, 3\dots$. If $k^{(m-1)}(0)\neq 0$, $\lambda=m$ is a double pole; otherwise, $m$ is a removable singularity in accordance with Theorem-\ref{ako}. For $\mathrm{Re}(\lambda)\geq 1$  and $\lambda\neq 1, 2, 3, \dots$, the finite-part coincide with the value of the analytic continuation there. Then the finite-part integral assumes the contour integral representation \eqref{n1regular}.
	
	Now let $\lambda=m$ and $k^{(m-1)}(0)\neq 0$. Under this condition $\lambda=m$ is a double pole of the analytic continuation. We write the analytic continuation in the form
	\begin{equation}\label{mimi}
		\mathcal{M}^*_a[k(t)\ln t;1-\lambda] = F_1(\lambda)+F_2(\lambda)
	\end{equation}
	where
	\begin{equation}\label{mimi2}
		F_1(\lambda)= \frac{1}{\left(e^{-2\pi\lambda i}-1\right)} \int_C \frac{k(t)\log z}{z^{\lambda}}\,\mathrm{d}z,\;\; F_2(\lambda)= \frac{-2\pi i e^{-2\pi\lambda i}}{\left(e^{-2\pi\lambda i}-1\right)^2} \int_C \frac{k(t)}{z^{\lambda}}\,\mathrm{d}z.
	\end{equation}
	Utilizing the linearity of the regularized limit, the regularized limit is given by
	\begin{equation}\label{bebot}
		\reglim{\lambda}{m}\mathcal{M}^*_a[k(t)\ln t;1-\lambda] = \reglim{\lambda}{m}F_1(\lambda) + \reglim{\lambda}{m}F_2(\lambda) .
	\end{equation}
	
	We now compute separately the limits. First for $F_1(\lambda)$. We rationalize $F_1(\lambda)$ in the form 
	\begin{equation*}
		F_1(\lambda)=\frac{f_1(\lambda)}{g_1(\lambda)},
	\end{equation*}
	where
	\begin{equation}
		f_1(\lambda)=\int_C \frac{k(z)\log z}{z^{\lambda}}\,\mathrm{d}z, \;\;\; g_1(z)= \left(e^{-2\pi\lambda i }-1\right) .
	\end{equation}
	Now $f_1(m)$ cannot vanish so that $\lambda=m$ is a simple pole of $F_1(\lambda)$. The regularized limit is then given by equation \eqref{rv0}. Performing the indicated differentiations at $\lambda=m$, we obtain the regularized limit 
	\begin{eqnarray}\label{reg1}
		\reglim{\lambda}{m}\frac{f_1(\lambda)}{g_1(\lambda)} 
		= \frac{1}{2\pi i} \int_C \frac{k(t)\log^2 z}{z^m}\,\mathrm{d}z - \frac{1}{2}\int_C \frac{k(z)\log z}{z^m}\,\mathrm{d}z , 
	\end{eqnarray}
	for positive integers $m$. 
	
	Now we compute for the regularized limit for $F_2(\lambda)$. We rationalize $F_2(\lambda)$ in the form
	\begin{equation*}
		F_2(\lambda)=\frac{f_2(\lambda)}{g_2(\lambda)},
	\end{equation*}
	where
		\begin{equation}
		f_2(\lambda)=-2\pi i e^{-2\pi i\lambda} \int_C \frac{k(z)}{z^{\lambda}}\,\mathrm{d}z, \;\;\; g_2(\lambda)= (e^{-2\pi \lambda i}-1)^2 .
	\end{equation}
	Since $\lambda=m$ is a double pole, the regularized limit, according to equation \eqref{reglimdouble}, is given by 
	\begin{equation}
		\begin{split}
		\reglim{\lambda}{m} \frac{f_2(\lambda)}{g_2(\lambda)} &= \frac{f_2''(m)}{g_2''(m)}  -\frac{2}{3} \frac{f_2'(m) g_2'''(m)}{(g_2''(m))^2}\\
		&\hspace{12mm} + \frac{f_2(m) \left(4 (g_2'''(m))^2 - 3 g_2''(m) g_2''''(m)\right)}{18 (g_2''(m))^3} . 
		\end{split}
	\end{equation}
	Performing the indicated differentiations at $\lambda=m$, we obtain
	\begin{eqnarray}
		\reglim{\lambda}{m} \frac{f_2(\lambda)}{g_2(\lambda)} 
		= \frac{1}{12 (2\pi i)} \int_C \frac{k(z)}{z^m}\,\mathrm{d}z - \frac{1}{2(2\pi i)} \int_C \frac{k(z)\log^2 z}{z^m}\,\mathrm{d}z.\label{reg2}
	\end{eqnarray}
	
	Finally we compute for the regularized limit \eqref{bebot} by adding equations \eqref{reg1} and \eqref{reg2}. The result is
	\begin{equation}\label{keke2}
		\reglim{\lambda}{m}\mathcal{M}^*_a[k(t)\ln t;1-\lambda] = \frac{1}{2\pi i} \int_C \frac{k(t)}{z^m}\left[\frac{1}{2}\log^2 z - \pi i \log z - \frac{\pi^2}{3}\right]\,\mathrm{d}z,
	\end{equation}
	for $k^{(m-1)}(0)\neq 0 $. Theorem-\ref{bek} asserts that \eqref{keke2} is equal to finite-part integral. Then the right hand side of \eqref{keke2} provides the contour integral representation of the finite-part integral. This proves \eqref{fpi2c} for $k^{(m-1)}(0)\neq 0$ or at the double poles of the analytic continuation of the Mellin transform. 
	
	If it happens that $k^{(m-1)}(0)=0$, the point $\lambda=m$ is a removable singularity according to Theorem-\ref{ako}. This can be seen by inspection of the expression \eqref{helo} for the analytic continuation. We may use the Cauchy limit to obtain the desired regularized limit; however, it is not convenient to do so. We exploit again the linearity of the regularized limit on the decomposition \eqref{mimi} only this time $F_2(\lambda)$ has a simple pole at $\lambda=m$. The regularized limit for $F_1(\lambda)$ is already given by \eqref{reg1}. For $F_2(\lambda)$, we rationalize it as
	\begin{equation}
		F_2(\lambda)= \frac{\tilde{f}_2(\lambda)}{\tilde{g}_2(\lambda)},
	\end{equation}  
	where
	\begin{equation}
		\tilde{f}_2(\lambda) = \frac{1}{(e^{-2\pi \lambda i}-1)} \int_{C} \frac{k(z)}{z^{\lambda}}\,\mathrm{d}z ,\;\;\; \tilde{g}_2(\lambda)=\frac{(e^{-2\pi\lambda i}-1)}{-2\pi i e^{-2\pi\lambda i}},
	\end{equation}
	in which $\tilde{f}_2^{(n)}(m)=\lim_{\lambda\rightarrow m} \tilde{f}_2^{(n)}(\lambda)$ for all non-negative integer $n$. In this form, the simple pole singularity of $F_2(\lambda)$ is clearly manifested. 
	
	Using equation \eqref{reglimsimple} for simple poles and performing the indicated differentiations at $\lambda=m$, we obtain 
	\begin{equation}\label{reg22}
		\reglim{\lambda}{m}\frac{\tilde{f}_2(\lambda)}{\tilde{g}_2(\lambda)} = -\frac{1}{4\pi i} \int_C \frac{k(z) \log^2 z}{z^m}\,\mathrm{d}z 
	\end{equation}
Combining equations \eqref{reg1} and \eqref{reg22}, we arrive at the desired regularized limit 
	\begin{equation}\label{xixi}
		\reglim{\lambda}{m}\mathcal{M}^*_a[k(t)\ln t;1-\lambda] = \frac{1}{2\pi i} \int_C \frac{k(t)}{z^m}\left[\frac{1}{2}\log^2 z - \pi i \log z\right]\,\mathrm{d}z,
	\end{equation}
	for $k^{(m-1)}(0)=0$. Observe that \eqref{xixi} is just the reduction of equation \eqref{fpi2c} when $\int_Cz^{-m}k(z)\mathrm{d}z = 0$ or when $k^{(m-1)}(0)=0$. Since the Cauchy and the regularized limit coincide at removable singularities, equation \eqref{fpi2c} holds for pole and removable singularities or for all positive integer $m$. 
	
\end{proof}

\section*{Part III}
\section{Finite-part Integration of the Stieltjes Transform in the Absence of Logarithmic Singularities}\label{nonlogcase}
In this Section we start to address the second of the two main problems of the paper---the evaluation of the Stieltjes transform by finite-part integration. Our solution proceeds in two steps. First, is to evaluate the Stieltjes transform in the non-logarithmic case using the machinery of finite-part integration; and, second, is to evaluate the transform in the presence of logarithmic singularities at the origin by repeated differentiation and application of the regularized limit on the result for the non-logarithmic case. Here we built the foundation upon which the solution to the Stieltjes transform in the general case can be obtained. While our problem in this Section has been dealt with in \cite{galapon2}, our method of proof here is more powerful than the one used in \cite{galapon2}. Our method here can be applied to the generalized Stieltjes transform to which the method of proof in \cite{galapon2} cannot be applied directly.

Now finite-part integration generally proceeds in two major steps. The first step is deliberately inducing divergent integrals in the given integral, and the second step is recasting the integral in a form that leads into its evaluation in terms of the finite-parts of the induced divergent integrals. The first step involves identifying the divergent integrals, extracting the corresponding finite-part integrals, and obtaining the contour integral representation of the finite-parts. The second step involves extracting the given integral from a contour integral whose functional form is dictated by the contour integral representation of the finite-part integrals; this is followed by the desired term by term integration. The sinew that connects the two steps is the representation of the finite-part integral as a contour integral in the complex plane. 

Here, for every $k(t)$ in $\mathcal{K}_a$, we perform finite-part integration on the Stieltjes transform 
\begin{equation}\label{lnst}
	\int_0^a \frac{k(t) }{t^{\nu}(\omega+t)}\,\mathrm{d}t,
\end{equation}
for $0\leq\mathrm{Re}(\nu)<1$, $|\mathrm{Arg}(\omega)|<\pi$ and $0<a\leq\infty$. To deliberately induce divergent integrals, we introduce the expansion \eqref{bino} back into \eqref{lnst}. The divergent integrals are 
\begin{equation}
	\int_0^a \frac{k(t)}{t^{\nu+k+1}}\,\mathrm{d}t ,
\end{equation}
for all non-negative integer $k$. The appropriate contour integral representation of the finite-part of these divergent integrals is dictated by $\nu$. The case $0< \mathrm{\nu}<1$ corresponds to $\lambda=\nu+k+1$ which is a regular point of the analytic continuation of the Mellin transform; then the appropriate contour integral representation is given by \eqref{que}. On the other hand, the case $\nu=0$ corresponds to $\lambda=k+1$ which is an isolated singularity of the the analytic continuation; then the appropriate contour integral representation is given by \eqref{poleIso}. 

We do not loose generality by assuming that $k(0)\neq 0$. If it happens that $k(0)=0$, then $k(z)=z^m g(z)$ for some positive integer $m$, and $g(z)$ is some analytic function in the interval satisfying $g(0)\neq 0$. Now we expand
\begin{equation}
	\frac{1}{\omega+t} = \sum_{j=0}^{m-1} \frac{(-1)^j \omega^j}{t^{j+1}} + \frac{(-1)^m \omega^m}{t^m (\omega+t)}.
\end{equation} 
Substituting this back into the integral yields
\begin{equation}\label{lnst3}
	\int_0^a \frac{k(t)}{t^{\nu}(\omega+t)}\,\mathrm{d}t = \sum_{j=0}^{m-1} (-1)^j \omega^j \int_0^a t^{m-j-1} g(t) \,\mathrm{d}t +(-1)^m \omega^m \int_0^a \frac{g(t) }{\omega+t}\,\mathrm{d}t . 
\end{equation}
The first term in the right hand side of the equation exists; now the second term involves an integral in the form of the desired Stieltjes transform because $g(0)\neq 0$.

\subsection{$\nu\neq 0$}
\begin{theorem}\label{miyo}
	Let $k(t)$ be in $\mathcal{K}_a$ and $k(z)$ its complex extension. If $\rho_0$ is the distance of the singularity of $k(z)$ nearest to the origin, then 
	\begin{equation}\label{orig1}
		\int_0^a \frac{k(t)}{t^{\nu} (\omega+t)}\,\mathrm{d}t = \sum_{j=0}^{\infty} (-1)^j \omega^j \bbint{0}{a} \frac{k(t)}{x^{j+\nu+1}}\,\mathrm{d}t +\frac{\pi k(-\omega)}{\omega^{\nu}\sin(\pi\nu)} ,\; 
	\end{equation}
for all $\omega$ satisfying $|\omega|<\mathrm{min}(\rho_0,a)$, $|\operatorname{Arg}\omega|<\pi$,  $0<\mathrm{Re}(\nu)<1$, where $\omega^{-\nu}$ takes its principal value. If the Stieltjes integral exists as $a\rightarrow\infty$, then equality \eqref{orig1} also holds for $a=\infty$ for all $|\omega|<\rho_0$ when $k(z)$ has at least one singularity or for all $|\omega|<\infty$ when $k(z)$ happens to be entire.

\end{theorem}
\begin{proof}
	We assume in the mean time that $\omega>0$ and later effect analytic continuation to cover the general case. We first consider the case of finite $a$ and $\rho_0$. For $\nu\neq 0$, the finite-part integrals have the contour integral representation
	\begin{equation}\label{ilo}
		\begin{split}
			\bbint{0}{a}\frac{k(t) }{t^{\nu+k+1}}\,\mathrm{d}t = \frac{1}{(e^{-2\pi(\nu+k+1)i}-1)} \int_{C}\frac{k(z) }{z^{\nu+k+1}}\, \mathrm{d}z,
		\end{split}
	\end{equation}
	coming from \eqref{que} at regular points of the analytic continuation. The contour $C$ must satisfy the following conditions: ({\it i}) the singularity of the kernel, $z=-\omega$, stays to the left when $C$ is traversed in the same direction; ({\it ii}) all singularities of $k(z)$ must stay to the right when the contour $C$ is traversed in the positive sense. The former will allow us to perform term by term integration, and the later will allow us to identify contour integrals arising from the term by term integration as finite-part integrals.
	
	We proceed by extracting the Stieltjes transform from a contour integral that is consistent with the contour integral representation of the finite-part integrals. From the contour integral representation \eqref{ilo} of the finite-part integral, the desired contour integral to extract the Stieltjes integral from is
	\begin{equation}
		\frac{1}{(e^{-2\pi\nu i}-1)}\int_{\mathrm{C}} \frac{k(z)}{z^{\nu} (\omega+z)} \,\mathrm{d}z.
	\end{equation}
	Deforming the contour ${C}$ to the contour ${C}'$, the Stieltjes transform assumes the representation
	\begin{equation}\label{mi}
		\begin{split}
			&\int_0^a \frac{k(t)}{t^{\nu} (\omega+t)}\,\mathrm{d}t = \frac{1}{(e^{-2\pi\nu i}-1)}\int_{\mathrm{C}} \frac{k(z)}{z^{\nu} (\omega+z)} \,\mathrm{d}z +\frac{\pi k(-\omega)}{\omega^{\nu} \sin(\pi\nu)},
		\end{split}
	\end{equation}
	where the second term is the residue contribution from the simple pole $z=-\omega$ of the kernel of the transformation.
	
	The next step is to implement term by term integration by performing an expansion of the kernel about $\omega=0$, 
	\begin{equation}\label{expansion}
		\frac{1}{\omega+z} = \sum_{l=0}^{\infty} (-1)^l \frac{\omega^l}{z^{l+1}} ,
	\end{equation}
	which is valid provided $\omega<|z|$. We choose the contour ${C}$ such that $\omega<|z|$ for all $z$ in the contour ${C}$; this implies condition ({\it i}). Under this condition, the expansion \eqref{expansion} converges uniformly along $C$: If $d_0>\omega$ is the distance of the point $z_0$ in $C$ closest to the origin, then $|\omega/z|^k\leq |\omega/d_0|^k$ for all $k$ and for every $z$ in $C$; since $d_0>\omega$, the series $\sum_{k=0}^{\infty}(\omega/d_0)^k$ converges, implying that the expansion \eqref{expansion} uniformly converges in the entire length of $C$. This allows us to introduce the expansion \eqref{expansion} back into \eqref{mi} and perform term-by-term integration to yield
	\begin{equation}\label{rep2}
		\begin{split}
			&\int_0^a \frac{k(t)}{t^{\nu} (\omega+t)}\,\mathrm{d}t = \sum_{l=0}^{\infty}(-1)^l \omega^l \frac{1 z}{(e^{-2\pi\nu i}-1)}\int_{\mathrm{C}} \frac{k(z)}{z^{\nu +l+1} } \,\mathrm{d}z +\frac{\pi k(-\omega)}{\omega^{\nu} \sin(\pi\nu)} .
		\end{split}
	\end{equation} 
	Under condition ({\it ii}), the contour $C$ does not enclose any of the singularities of $k(z)$, so that the contour integrals in \eqref{rep2} are necessarily finite part integrals. We are led to the equality
	\begin{equation}\label{series}
		\begin{split}
			&\int_0^a \frac{k(t)}{t^{\nu} (\omega+t)}\,\mathrm{d}t = \sum_{l=0}^{\infty}(-1)^l \omega^l\bbint{0}{a}\frac{k(t)}{t^{\nu+l+1}}\mathrm{d}t +\frac{\pi k(-\omega)}{\omega^{\nu} \sin(\pi\nu)} .
		\end{split}
	\end{equation}
	Since $a$ falls in the contour $C$, it is necessary that $\omega<a$; also since all the singularities of $k(z)$ are to the right of $C$, it is also necessary that $\omega<\rho_0$; hence $\omega<\mathrm{min}(a,\rho_0)$. 
		
	We now show that the the infinite series of finite-parts in \eqref{series} is absolutely convergent.  First we establish a bound for the finite-parts using the contour integral representation given by equation \eqref{ilo}. Again let $d_0$ be the distance of the point $z_0$ in $C$ that is closest to the origin. Deforming the contour $C$ to $C'$ with the radius of the circle equal to $d_0$ leads to
	\begin{equation}\label{deform}
		\begin{split}
			\bbint{0}{a}\frac{k(t) }{t^{\nu+k+1}}\,\mathrm{d}t =& \int_{d_0}^a\frac{k(t)}{t^{\nu+k+1}}\mathrm{d}t+\frac{1}{\left(e^{-2\pi\nu i}-1\right)}\int_{0}^{2\pi}\frac{k(d_0 e^{i\theta}) }{(d_0e^{i\theta})^{\nu+k+1}}\,  i d_0 e^{i\theta} \mathrm{d}\theta ,
		\end{split}
	\end{equation}
	from which the following bound can be readily derived,
	\begin{equation}\label{ineq24}
		\begin{split}
			\left|\bbint{0}{a}\frac{k(t) }{t^{\nu+k+1}}\,\mathrm{d}t \right| \leq \frac{M(a)}{|d_0^{\nu}|d_0^{k}}  .
		\end{split}
	\end{equation}
where
\begin{equation}\label{ineq23}
	\begin{split}
		M(a)=  \int_{d_0}^a \frac{\left|k(t)\right|}{t}\,\mathrm{d}t + \int_0^{2\pi} \left|\frac{k(d_0 e^{i\theta})}{\left(e^{-2\pi\nu i}-1\right)} i d_0 e^{i\theta} \right|\mathrm{d}\theta ,
	\end{split}
\end{equation}
	for all non-negative integer $k$. Then we arrive at the inequality 
	\begin{equation}\label{ineq}
		\left|\sum_{k=0}^{\infty} (-1)^k \omega^k \bbint{0}{a} \frac{k(t) }{t^{\nu+k+1}}\,\mathrm{d}t\right|\leq  \frac{M(a)}{|d_0^{\nu}|} \sum_{k=0}^{\infty}\left(\frac{\omega}{d_0}\right)^k .
	\end{equation}
	The sum in the right hand side converges whenever $\omega<d_0$.  Now if $\omega<\mathrm{min}(a,\rho_0)$, there always exists a $d_0$ such that $\omega<d_0<\mathrm{min}(a,\rho_0)$. This implies that the left hand side of \eqref{ineq} is finite or the infinite series of finite-parts is absolutely convergent. 
	
	The preceding result holds for finite $a$ only. We now show that equation \eqref{orig1} holds for infinite $a$ as well. By hypothesis the Stieltjes transform exists as $a\rightarrow\infty$, which implies that $t^{-1} k(t)$ is integrable at infinity. Then, from the integral representation of the finite-part integral for $a<\infty$ given by equation \eqref{deform}, the limit as $a\rightarrow\infty$ is seen to exist as a consequence of the existence of the Stieltjes transform. Then the finite-part integral 
	\begin{equation}\label{limfpi}
		\bbint{0}{\infty} \frac{k(t)}{t^{\nu+k+1}}\,\mathrm{d}t = \lim_{a\rightarrow\infty} \bbint{0}{a} \frac{k(t)}{t^{\nu+k+1}}\,\mathrm{d}t
	\end{equation}
	exists for all non-negative integer $k$. Now inequality \eqref{ineq} implies that the series in the right hand side of \eqref{series} is uniformly convergent  for all $a$ in the open interval $(0,\infty)$. The existence of the limit \eqref{limfpi} and the uniform convergence of the series \eqref{series} for all finite $a$ imply the equality
	\begin{equation}
		\lim_{a\rightarrow\infty} \sum_{k=0}^{\infty} (-1)^k \omega^k \bbint{0}{a} \frac{k(t) }{t^{\nu+k+1}}\,\mathrm{d}t= \sum_{k=0}^{\infty} (-1)^k \omega^k \lim_{a\rightarrow\infty} \bbint{0}{a} \frac{k(t)}{t^{\nu+k+1}}\,\mathrm{d}t .
	\end{equation}
	This means that the evaluation \eqref{orig1} of the Stieltjes transform holds for infinite $a$ as well. Since $a$ is already infinite, the condition for convergence of the series is $\omega<\rho_0$ when $k(z)$ has at least one singularity or $\omega<\infty$ when $k(z)$ happens to be entire. This completes the proof of the equality \eqref{series} for all positive $\omega<\mathrm{min}(a,\rho_0)$.  
	
	We now extend the validity of equation \eqref{series} for complex $\omega$ by analytic continuation. The left hand side of \eqref{series} is analytic for all complex $\omega$ provided $|\operatorname{Arg}(\omega)|<\pi$. Let us consider the right hand side. The inequality \eqref{ineq} holds as well for all complex $\omega$ satisfying $|\omega|<\operatorname{min}(a,\rho_0)$; then the infinite series in the right hand side holds for all such $\omega$. Now the second term of the right hand side is analytic for all complex $\omega$ satisfying, at least, $|\omega|<\rho_0$, and $|\operatorname{Arg}(\omega)|<\pi$ when $\omega^{\nu}$ is replaced with its the principal value. Then the entire right hand side is analytic for all complex $\omega$ with $|\omega|<\operatorname{min}(a,\rho_0)$ and  $|\operatorname{Arg}(\omega)|<\pi$. Since both sides of \eqref{series} are equal for all positive $\omega<\operatorname{min}(a,\rho_0)$, then, by the principle of analytic continuation, the equality extends in the entire common domain of analyticity of both sides of the equation, which is the domain consisting of all complex $\omega$ with $|\omega|<\operatorname{min}(a,\rho_0)$ and  $|\operatorname{Arg}(\omega)|<\pi$. This completes the proof of the theorem. 
\end{proof}

\subsection{$\nu=0$}

\begin{theorem}
	Under the same relevant conditions as in Theorem-\ref{miyo},
	\begin{equation}\label{orig2}
		\int_0^a \frac{k(t)}{(\omega+t)}\,\mathrm{d}t = \sum_{j=0}^{\infty} (-1)^j \omega^j \bbint{0}{a} \frac{k(t)}{t^{j+1}}\,\mathrm{d}t -k(-\omega)\operatorname{Log}\omega . 
	\end{equation}
\end{theorem}
\begin{proof}
Again we initially let $\omega>0$. The relevant finite-parts assume the contour integral representation
\begin{equation}\label{lala}
	\bbint{0}{a} \frac{k(t)}{t^{k+1}}\,\mathrm{d}t = \frac{1}{2\pi i} \int_C \frac{k(z)}{z^{k+1}}\left(\log z - i\pi\right)\,\mathrm{d}t. 
\end{equation}
From \eqref{lala} we extract the Stieltjes transform from the contour integral
	\begin{equation}
		\frac{1}{2\pi i} \int_{\mathrm{C}} \frac{k(z)}{\omega+z} \left(\log z - i \pi \right)\,\mathrm{d}z ,
	\end{equation}
	where the contour $C$ satisfies the same conditions as in equation \eqref{ilo}. Again we then deform the contour $C$ into the contour $C'$, from which the given integral emerges,
	\begin{equation}\label{rep}
		\begin{split}
			\int_0^a \frac{k(t)}{\omega + t}\,\mathrm{d}t =& \frac{1}{2\pi i} \int_{\mathrm{C}} \frac{k(z)}{\omega+z} \left(\log z - i \pi 
			\right)\mathrm{d}z- k(-\omega) \ln\omega .
		\end{split}
	\end{equation} 
	The second term is the contribution coming from the simple pole of the kernel at $z=-\omega$. 
	We then proceed in the same manner as in proving \eqref{orig1} under the same relevant conditions. 	Finally, extending the result away from the real axis leads to replacing $\ln\omega$ with its principal value $\operatorname{Log}(\omega)$, with the restriction $|\operatorname{Arg}(\omega)|<\pi$.
\end{proof}

\section{Finite-part Integration of the Stieltjes Transform in the Presence of Arbitrary Logarithmic Singularities}\label{logcase}
Now we consider the problem of evaluating the Stieltjes transform \eqref{probhere}. Under the condition of uniform convergence of the series, the Stieltjes transform can be evaluated by distributing the integration over the summation,
\begin{equation}
	\int_0^a \frac{h(t)}{\omega+t}\,\mathrm{d}t = \sum_{k=0}^{\infty}\sum_{l=0}^{M(k)} \int_0^a \frac{k_{kl}(t) \ln^l t}{t^{\nu_k}(\omega+t)}\,\mathrm{d}t . 
\end{equation}
Then the problem reduces to evaluating Stieltjes integrals of the form
\begin{equation}\label{components2}
	\int_0^a \frac{k(t)\ln^n t}{t^{\nu}(\omega+t)}\,\mathrm{d}t,\;\;\; 0\leq \mathrm{Re}(\nu)<1 .
\end{equation}
We have already solved the case $n=0$ by explicit finite-part integration. Finite-part integration requires explicit form of the contour integral representations of the finite-part integrals involving powers of the logarithm which may be intractable to deal with. Here we evaluate \eqref{components2} using the result of the non-logarithmic case by repeated differentiation and application of the regularized limit.

\subsection{Case $\nu\neq 0$}
\begin{lemma} \label{xe}
	Let $k(t)$ be in $\mathcal{K}_a$. For all non-negative integer $s$ and positive integer $r$,
	\begin{equation}\label{bo}
		\frac{\mathrm{d}}{\mathrm{d}\nu}\bbint{0}{a}\frac{k(t) \ln^s t}{t^{\nu+r}}\,\mathrm{d}t=-\bbint{0}{a}\frac{k(t) \ln^{s+1} t}{t^{\nu+r}}\,\mathrm{d}t
	\end{equation}
for $0<\mathrm{Re}(\nu)<1$.
\end{lemma}

\begin{proof}
Since $\nu+r$ is not an integer, we have from Corollary-\ref{colloanal} the representation for the relevant finite-part integral,
\begin{equation}\label{be}
	\begin{split}
	\bbint{0}{a}\frac{k(t) \ln^s t}{t^{\nu+r}}\,\mathrm{d}t =& \int_{\epsilon}^a \frac{k(t) \ln^s t}{t^{\nu+r}}\,\mathrm{d}t \\
	&+ (-1)^s s! \sum_{j=0}^s \frac{(-1)^j \ln^j\epsilon}{j!} \sum_{l=0}^{\infty} a_l \frac{\epsilon^{l-\nu-r+1}}{(l-\nu-r+1)^{s-j+1}} .
	\end{split}
\end{equation}
The first term in the right hand side can be differentiated inside the integral and the series can also be differentiated term by term because it is uniformly convergent in $\epsilon$. Differentiating both sides of equation \eqref{be} yields
 \begin{equation}
 	\begin{split}
 		\frac{\mathrm{d}}{\mathrm{d}\nu}\bbint{0}{a}\frac{k(t) \ln^s t}{t^{\nu+r}}\,\mathrm{d}t =&- \int_{\epsilon}^a \frac{k(t) \ln^{s+1} t}{t^{\nu+r}}\,\mathrm{d}t 
 		+ (-1)^{s+1} s! \sum_{l=0}^{\infty} a_l \epsilon^{l-\nu-r+1} \\
 		&\hspace{-20mm}\times \left[\sum_{j=0}^s \frac{(-1)^j \ln^{j+1}\epsilon}{j! (l-\nu-r+1)^{s-j+1}} - \sum_{j=0}^s \frac{(-1)^j(s+1-j)\ln^j\epsilon}{j! (l-\nu-r+1)^{s+1-j+1}} \right].
 	\end{split}
 \end{equation}
Shifting the index in the first sum from $j$ to $j+1$, we find it to be canceled by the $j\ln^j\epsilon$ term in the second sum and only the term $(s+1)\ln^j\epsilon$ term survives. Then
 \begin{equation}
	\begin{split}
		\frac{\mathrm{d}}{\mathrm{d}\nu}\bbint{0}{a}\frac{k(t) \ln^s t}{t^{\nu+r}}\,\mathrm{d}t =&- \int_{\epsilon}^a \frac{k(t) \ln^{s+1} t}{t^{\nu+r}}\,\mathrm{d}t\\
		&\hspace{-25mm}- (-1)^{s+1} (s+1)! \sum_{j=0}^{s+1}\frac{(-1)^j \ln^j \epsilon}{j!}
		\sum_{l=0}^{\infty} a_l \frac{\epsilon^{l-\nu-r+1}}{(l-\nu-r+1)^{(s+1)-j+1}}. 
	\end{split}
\end{equation}
Comparing the right hand side with Corollary-\ref{colloanal} we obtain the desired equality \eqref{bo}. 
 
\end{proof}

\begin{theorem} \label{hi}
Let $k(t)$ be in $\mathcal{K}_a$ and $k(z)$ be its complex extension. If $\rho_0$ is the distance of the singularity of $k(z)$ nearest to the origin, then 
\begin{equation}\label{bebe}
	\int_0^a\frac{k(t) \ln^n t}{t^{\nu}(\omega+t)}\,\mathrm{d}t = \sum_{j=0}^{\infty} (-1)^j \omega^j \bbint{0}{a} \frac{k(t)\ln^n t}{t^{j+\nu+1}}\,\mathrm{d}t + k(-\omega) \Delta_n(\nu,\omega)
\end{equation} 
for all $\omega$ satisfying $|\omega|<\mathrm{min}(\rho_0,a)$, $|\operatorname{Arg}\omega|<\pi$, for all non-negative integer $n$, and $0<\mathrm{Re}(\nu)<1$, where
\begin{equation}\label{wa1}
	\Delta_n(\nu,\omega)= (-1)^n \frac{\pi}{\omega^{\nu}}\sum_{l=0}^n (-1)^l \binom{n}{l} (\operatorname{Log} \omega)^{n-l}\, D_l(\nu),
\end{equation}
\begin{equation}\label{wa2}
	\begin{split}
		D_l(\nu)=&(-1)^{\left[l/2\right]} \pi^l \csc(\pi\nu) \sum_{k=1}^n\frac{1}{2^k}\sum_{m=0}^k (-1)^m \binom{k}{m}(2m+1)^l\\
		&\times\sum_{p=0}^{\left[(k-\gamma)/2\right]} (-1)^p \binom{k}{2p+\gamma} \cot^{2p+\gamma}(\pi\nu),\;\;\; \gamma=\frac{1-(-1)^l}{2}. 
	\end{split}
\end{equation}
If the Stieltjes integral exists as $a\rightarrow\infty$, then equality \eqref{bebe} also holds for $a=\infty$ for all $|\omega|<\rho_0$ when $k(z)$ has at least one singularity or for all $|\omega|<\infty$ when $k(z)$ happens to be entire. Moreover, it holds that 
\begin{equation}\label{asym1}
	\int_0^a \frac{k(t)\ln^n t}{t^{\nu}(\omega+t)}\,\mathrm{d}t \sim k(0) \Delta_n(\nu,\omega),\;\; \omega\rightarrow 0 .
\end{equation}
\end{theorem}

\begin{proof} First, we consider the $a<\infty$ case. Under the stated conditions, the integral $\int_0^a t^{-\nu} (\omega+t)^{-1} k(t)\ln^n t\,\mathrm{d}t$ is uniformly convergent with respect to the parameter $\nu$. Then we can perform repeated differentiation with respect to $\nu$ inside the integral to obtain
	\begin{equation}\label{po}
		\int_0^a\frac{k(t) \ln^n t}{t^{\nu}(\omega+t)}\,\mathrm{d}t=(-1)^n \frac{\mathrm{d}^n}{\mathrm{d}\nu^{n}} \int_0^a\frac{k(t)}{t^{\nu}(\omega+t)}\,\mathrm{d}t .
	\end{equation}
	We then substitute the non-logarithmic case result \eqref{orig1} in the right hand side of \eqref{po} under the conditions stated in Theorem-\ref{miyo}. The infinite series in \eqref{orig1} is uniformly convergent with respect to $\nu$ in the full range $0<\mathrm{Re}(\nu)<1$; this follows from the bound \eqref{ineq}. Then we can perform term by term differentiation on the right hand side of \eqref{orig1}, 
	\begin{equation}\label{orig1diff}
		\begin{split}
		\int_0^a \frac{k(t) \ln^n t}{t^{\nu} (\omega+t)}\,\mathrm{d}t =& \sum_{j=0}^{\infty} (-1)^j \omega^j (-1)^n \frac{\mathrm{d}^n}{\mathrm{d}\nu^n} \bbint{0}{a} \frac{k(t)}{x^{j+\nu+1}}\,\mathrm{d}t\\
		& + f(-\omega)\pi(-1)^n \frac{\mathrm{d}^n}{\mathrm{d}\nu^n}\left(\omega^{-\nu}\csc(\pi\nu)\right)  .
		\end{split}
	\end{equation}

Given that the summation index $j$ in \eqref{orig1diff} is a non-negative integer, Lemma-\ref{xe} holds. By repeated use of the Lemma, the infinite series in \eqref{orig1diff} assumes the form
\begin{equation}\label{mi1}
    \sum_{j=0}^{\infty} (-1)^j \omega^j (-1)^n \frac{\mathrm{d}^n}{\mathrm{d}\nu^n} \bbint{0}{a} \frac{k(t)}{x^{j+\nu+1}}\,\mathrm{d}t=\sum_{j=0}^{\infty} (-1)^j \omega^j \bbint{0}{a} \frac{k(t)\ln^n t}{x^{j+\nu+1}}\,\mathrm{d}t ;
\end{equation}
by the uniform convergence of the series in \eqref{orig1}, this series converges as well under the same conditions as those in Theorem-\ref{miyo}, which are the same conditions on the current Theorem under consideration. Now by Leibniz rule
\begin{equation}\label{mi2}
    \frac{\mathrm{d}^n}{\mathrm{d}\nu^n}\left(\omega^{-\nu}\csc(\pi\nu)\right)= \frac{1}{\omega^{\nu}}\sum_{l=0}^n (-1)^l \binom{n}{l} (\operatorname{Log} \omega)^{n-l}\frac{\mathrm{d}^l}{\mathrm{d}\nu^l}\csc(\pi\nu).
\end{equation}
The derivative $D_{\nu}^l[\csc(\pi\nu)]$ is precisely the constant $D_l(\nu)$ for all positive integer $l$ \cite[pg. 8, \#9]{yury}. Substituting \eqref{mi1} and \eqref{mi2} back into equation \eqref{orig1diff} yields \eqref{bebe}.

We have appealed to uniform convergence of the result for the non-logarithmic case to arrive at our preceding conclusion. However, we wish now to establish convergence of the series in \eqref{bebe} directly by working on the series \eqref{mi1} itself. For the the case of $a<\infty$ there are two special cases, $\rho_0<a$ and $\rho_0>a$. We now do the $\rho_0<a$ case. Foremost, we establish a bound for the finite-part integrals that appear in the summation. For some positive $\epsilon<\rho_0$, we have from \ref{xoxo2}, 
\begin{equation}\label{hoho}
	\begin{split}
	\bbint{0}{a} \frac{k(t)\ln^n t}{t^{r+\nu+1}}\,\mathrm{d}t=&\int_{\epsilon}^a \frac{k(t)\ln^n t}{t^{r+\nu+1}}\,\mathrm{d}t \\
	 &\hspace{-12mm}+ (-1)^n n! \sum_{l=0}^{\infty} a_l \epsilon^{l-r-\nu}\sum_{j=0}^n \frac{(-1)^j}{j!} \frac{\ln^j\epsilon}{(l-r-\nu)^{n-j+1}} ,
	\end{split}
\end{equation}
with the replacement $\lambda=j+\nu+1$. We have the bound
\begin{equation}\label{belat}
	\begin{split}
		\left|\bbint{0}{a} \frac{k(t)\ln^n t}{t^{r+\nu+1}}\,\mathrm{d}t\right|\leq &\frac{1}{\epsilon^{r+\nu}}\int_{\epsilon}^a \frac{\left|k(t)\ln^n t\right|}{t}\,\mathrm{d}t \\
		&\hspace{-12mm}+ \frac{ n!}{\epsilon^{r+\nu}} \sum_{l=0}^{\infty} \left|a_l\right| \epsilon^{l}\sum_{j=0}^n \frac{1}{j!} \frac{\left|\ln\epsilon\right|^j}{|l-r-\nu|^{n-j+1}} ,
	\end{split}
\end{equation}
where the first term follows from the inequality \begin{equation}\label{ne}
	\int_{\epsilon}^a \frac{ \left|k(t) \ln^n t\right|}{t^{r+\nu+1}}\mathrm{d}t\leq \frac{1}{\epsilon^{r+\nu}}\int_{\epsilon}^a \frac{\left| k(t) \ln^n t\right|}{t}\mathrm{d}t . 
	\end{equation}
	
We now wish to disentangle the double sum in inequality \eqref{belat} and obtain a bound independent of the indexes $l$ and $r$. We do so with the replacement
\begin{equation}\label{bolat}
	\frac{1}{(l-r-\nu)}=\frac{i}{(e^{-2\pi \nu i}-1)}\int_0^{2\pi} e^{i (l-r-\nu)\theta}\mathrm{d}\theta .
\end{equation}
Taking the modulus of both sides of \eqref{bolat}, we obtain the bound
\begin{equation}\label{bound1}
	\frac{1}{|l-r-\nu|}\leq\frac{2\sinh\left(\pi \mathrm{Im}(\nu)\right)}{\mathrm{Im}(\nu) \sqrt{\sin^2\!\left(\pi\mathrm{Re}(\nu)\right)+\sinh^2\!\left(\pi\mathrm{Im}(\nu)\right)}},\;\; \mathrm{Im}(\nu)\neq 0,
\end{equation}
\begin{equation}\label{bound2}
	\frac{1}{|l-r-\nu|}\leq\frac{2\pi}{ |\sin\left(\pi\nu\right)|},\;\; \nu\neq 0,\;\; \mathrm{Im}(\nu) = 0 .
\end{equation}
We collectively denote $M(\nu)$ the right hand sides of \eqref{bound1} and \eqref{bound2}. Notice that $M(\nu)$ is independent of $r$ and $l$.  Then we have the bound
\begin{equation}\label{akoito}
	\begin{split}
		\left|\bbint{0}{a} \frac{k(t)\ln^n t}{t^{r+\nu+1}}\,\mathrm{d}t\right|\leq \frac{M_{\nu}(a,\epsilon)}{\epsilon^{r+\nu}},
	\end{split}
\end{equation}
where
\begin{equation}\label{momo}
	\begin{split}
		M_{\nu}(a,\epsilon)=\int_{\epsilon}^a \frac{k(t)\ln^n t}{t}\,\mathrm{d}t+  n! \sum_{l=0}^{\infty} \left|a_l\right| \epsilon^{l}\sum_{j=0}^n  \frac{\left|\ln\epsilon\right|^j M(\nu)^{n-j+1}}{j!}.
	\end{split}
\end{equation}
The infinite series in the right hand side of \eqref{momo} converges since $\epsilon<\rho_0$.

We can now show that the infinite series converges absolutely. We have the inequality 
\begin{eqnarray}\label{ahah}
	\left|\sum_{r=0}^{\infty} (-1)^r \omega^r \bbint{0}{a} \frac{\left|k(t)\ln^n t\right|}{t^{r+\nu+1}}\,\mathrm{d}t\right|&\leq & \sum_{r=0}^{\infty} \left|\omega\right|^r \left|\bbint{0}{a} \frac{k(t)\ln^n t}{t^{r+\nu+1}}\,\mathrm{d}t\right| .
\end{eqnarray}
Substituting the bound \eqref{akoito} for the finite-part integrals yields
\begin{eqnarray}\label{nene}
	\left|\sum_{r=0}^{\infty} (-1)^r \omega^r \bbint{0}{a} \frac{\left|k(t)\ln^n t\right|}{t^{r+\nu+1}}\,\mathrm{d}t\right|
	\leq \frac{M_{\nu}(a,\epsilon)}{\epsilon^{\nu}} \sum_{r=0}^{\infty} \left|\frac{\omega}{\epsilon}\right|^r .
\end{eqnarray}
The right hand side of \eqref{nene} converges provided $|\omega|<\epsilon<\rho_0$. Now for every $\omega$ satisfying $|\omega|<\rho_0$ there always exists a positive $\epsilon$ satisfying $|\omega|<\epsilon<\rho_0$. This implies that the infinite series of finite-part integrals converges absolutely under the condition that $|\omega|<\rho_0<a$. 

Now we consider the $a<\rho_0$ case. This encompasses the situation where $k(z)$ is either entire or not. Under this condition, the integral in equation \eqref{hoho} can be evaluated explicitly by expanding $k(t)$ about $t=0$, followed by performing term by term integration. The bound is dominated by the term involving the upper limit of integration $a$ which is proportional $a^{-r}$. Substituting the bound back in inequality \eqref{ahah}, we find that the dominating term in the bound is proportional to
\begin{equation}
	\sum_{k=0}^{\infty} \left|\frac{\omega}{a}\right|^r .
\end{equation}
Thus in order for the sum to converge, it is necessary that $|\omega|<a$. This completes the proof that the condition $|\omega|<\mathrm{min}(\rho_0,a)$ is necessary for the infinite series to converge. 

Finally, we now consider the case for $a=\infty$. Under the hypothesis that the Stieltjes transform exists, it follows that $M_{\nu}(\infty,\epsilon)<\infty$. Then the finite-integral
\begin{equation}\label{xoxoxo}
	\bbint{0}{\infty}\frac{k(t)\ln^n(t)}{t^{r+\nu+1}}\,\mathrm{d}t = \lim_{a\rightarrow\infty}\bbint{0}{a}\frac{k(t)\ln^n(t)}{t^{r+\nu+1}}\,\mathrm{d}t
\end{equation}
exists for all $r$. Since $M_{\nu}(a,\epsilon)\leq M_{\nu}(\infty,\epsilon)$,  the  sum in \eqref{bebe} uniformly converges in $(0,\infty)$.  This, together with the existence of the limit \eqref{xoxoxo}, allows us to take the limit of both sides of equation \eqref{bebe} as $a\rightarrow\infty$, and interchange the summation and limit in the infinite series in the right hand side of equation \eqref{bebe}. This means that equation \eqref{bebe}  holds also when $a$ is replaced with infinity.

Finally, we establish the asymptotic relation \eqref{asym1}. From the exact evaluation of the Stieltjes transform given by equation \eqref{bebe}, we have
\begin{equation}\label{bebeasy}
	\int_0^a\frac{k(t) \ln^n t}{t^{\nu}(\omega+t)}\,\mathrm{d}t = \left(\bbint{0}{a} \frac{k(t)\ln^n t}{t^{\nu+1}}\,\mathrm{d}t + O(\omega) \right)+  \Delta_n(\nu,\omega) (k(0)+O(\omega)),\;\omega\rightarrow 0.
\end{equation}   
From \eqref{wa1} we have $\Delta_n(\nu,\omega)=O(\omega^{-\nu} \ln^n\omega)$ as $\omega\rightarrow 0$. Thus for positive $\mathrm{Re}(\nu)$, the second term dominates the first term for arbitrarily small $\omega$ and \eqref{asym1} follows.
\end{proof}

\subsubsection*{Example} For $n=1$ equation-\ref{bebe} reduces to the Stieltjes transform
\begin{equation}\label{sum2}
	\int_0^a \frac{k(t)\ln(t)}{t^{\nu}(\omega+t)}\,\mathrm{d}t = \sum_{k=0}^{\infty} (-1)^k \omega^k \bbint{0}{a} \frac{k(t) \ln(t)}{t^{\nu+k+1}}\,\mathrm{d}t + \frac{\pi k(-\omega)}{\omega^{\nu} \sin(\pi\nu)} \left(\pi\cot(\pi\nu) +\operatorname{Log}\omega\right), 
\end{equation} 
true under the conditions of Theorem-\ref{hi}. The asymptotic relation \eqref{asy1} follows from this expression for arbitrarily small $\omega$.

We can arrive at \eqref{sum2} by explicit finite-part integration using the contour integral representation of the finite-part integral given by \eqref{n1regular}. Under the same conditions as in Theorem-\ref{miyo}, we extract the right hand side of \eqref{sum2} from the contour integral 
\begin{equation}
	\int_{\mathrm{C}} \frac{k(z)}{z^{\nu} (\omega+z)} \left[\frac{\log z}{(e^{-2\pi\nu i}-1)} - \frac{2\pi i e^{-2\pi\nu i}}{(e^{-2\pi\nu i}-1)^2}\right]\,\mathrm{d}z. 
\end{equation}
Deforming the contour $C$ to the contour $C'$, the Stieltjes transform takes the representation
\begin{equation}\label{kiki}
	\begin{split}
		&\int_0^a \frac{k(t)\ln t}{t^{\nu} (\omega+t)}\,\mathrm{d}t = \int_{\mathrm{C}} \frac{k(z)}{z^{\nu} (\omega+z)} \left[\frac{\log z}{(e^{-2\pi\nu i}-1)} - \frac{2\pi i e^{-2\pi\nu i}}{(e^{-2\pi\nu i}-1)^2}\right]\,\mathrm{d}z\\
		& \hspace{34mm} +\frac{\pi k(-\omega)}{\omega^{\nu} \sin(\pi\nu)} (\pi\cot(\pi\nu) + \ln \omega),
	\end{split}
\end{equation}
assuming $\omega>0$. The second term is the residue contribution from the simple pole $z=-\omega$ of the kernel of the transformation. Introducing the expansion for $(\omega+z)^{-1}$ about $\omega=0$ back into \eqref{kiki} and distributing the integration, we arrive at \eqref{sum2} upon identifying the contour integrals as the finite-part integrals and replacing the multivalued functions with their principal values.

Explicit finite-part integration shows that the dominant term for small values of $\omega$ comes from the singularity of the kernel of the transformation, and this term is the missing term when the kernel is expanded and integrated term by term followed by naive assignment of the divergent integrals values equal to their finite-parts.

\subsection{Case $\nu=0$}
\begin{lemma} \label{mo}
Let $k(t)$ be in $\mathcal{K}_a$.   For all non-negative integer $n$ and positive integer $r$, 
	\begin{equation}\label{nono}
		\reglim{\nu}{0}\bbint{0}{a}\frac{k(t) \ln^n t}{t^{\nu+r}}\,\mathrm{d}t = \bbint{0}{a}\frac{k(t) \ln^n t}{t^{r}}\,\mathrm{d}t,
	\end{equation}
where $0<\mathrm{Re}(\nu)<1$. 
\end{lemma}
\begin{proof} Under the given conditions on $n$ and $r$, the domain of analyticity of the finite-part integral 
	\begin{equation}
		\bbint{0}{a} \frac{k(t)\ln^n t}{t^{\nu+r}}\,\mathrm{d}t,\nonumber 
	\end{equation}
	taken as a function of $\nu$, is $-\delta<\mathrm{Re}(\nu)$ for some $\delta>0$, which includes the relevant strip $0<\mathrm{Re}(\nu)<1$. In this domain, the limit $\nu\rightarrow 0$ exists and must equal the desired limit in \eqref{nono}. From equation \eqref{be} the finite-part diverges as $\nu\rightarrow 0$, as expected. This divergence occurs in the infinite series at the term $l=r-1$. Isolating this term and using the linearity of the regularized limit, we have
	\begin{equation}\label{no}
		\begin{split}
			\reglim{\nu}{0}\bbint{0}{a}\frac{k(t) \ln^s t}{t^{\nu+r}}\,\mathrm{d}t =& \reglim{\nu}{0}\int_{\epsilon}^a \frac{k(t) \ln^s t}{t^{\nu+r}}\,\mathrm{d}t + (-1)^s s! \sum_{j=0}^s \frac{(-1)^j \ln^j\epsilon}{j!}\\ &\hspace{-34mm}\times \left[\reglim{\nu}{0}\sum_{l=0,\,l\neq r-1}^{\infty} a_l \frac{\epsilon^{l-\nu-r+1}}{(l-\nu-r+1)^{s-j+1}} +a_{r-1} (-1)^{s-j+1} \reglim{\nu}{0}\frac{\epsilon^{-\nu}}{\nu^{s-j+1}}\right].
		\end{split}
	\end{equation}
The first two regularized limits reduce to the usual Cauchy limit. On the other hand, for the third limit, we have
\begin{equation}
	\reglim{\nu}{0}\frac{\epsilon^{-\nu}}{\nu^{s-j+1}}=(-1)^{s-j+1}\frac{ \ln^{s-j+1}\epsilon}{(s-j+1)!},
\end{equation}
on using \eqref{coco}. Substituting this back into \eqref{no} and comparing the result with Corollary-\ref{collozero}, we arrive at the desired equality \eqref{nono}.
\end{proof}

\begin{theorem}\label{hihi}
	Under the same relevant conditions as in Theorem-\ref{hi},
	\begin{equation}\label{kwa}
		\int_0^a \frac{k(t)\ln^n t}{(\omega+t)}\,\mathrm{d}t=\sum_{j=0}^{\infty}(-1)^j \omega^j \bbint{0}{a}\frac{k(t) \ln^n t}{t^{j+1}} + k(-\omega) \Delta_n(\omega)
	\end{equation}
	 where
	\begin{equation}\label{gen2}
		\Delta_n(\omega)=-\frac{\operatorname{Log}^{n+1}\omega}{n+1} + 2 \cdot n! \sum_{j=1}^{\lceil n/2\rceil} \frac{(2^{2j-1}-1) (-1)^{j} \pi^{2j} B_{2j}}{(2-2j+1)! (2j)!}\,(\operatorname{Log}\omega)^{n-2j+1} ,
	\end{equation}
in which the $B_{2j}$'s are the Bernoulli numbers. Moreover, it holds that
\begin{equation}\label{asym2}
	\int_0^a \frac{k(t) \ln^n t}{(\omega+t)}\,\mathrm{d}t \sim k(0) \Delta_n(\omega),\;\; \omega\rightarrow 0 .
\end{equation}
\end{theorem} 
\begin{proof} Again we first consider the case $a<\infty$. Under the stated conditions, the Stieltjes transform can be obtained from the Cauchy limit
\begin{equation}\label{oo}
    \int_0^a \frac{k(t)\ln^n t}{(\omega+t)}\,\mathrm{d}t=\lim_{\nu\rightarrow 0}\int_0^a \frac{k(t)\ln^n t}{t^{\nu}(\omega+t)}\,\mathrm{d}t,
\end{equation}
owing to the uniform convergence of the integral in the right hand side of \eqref{oo} in its strip of analyticity, which includes $\nu=0$ in its interior. We replace the Cauchy limit with the regularized limit and apply the result \eqref{bebe}. Using the linearity of the regularized limit, we arrive at
\begin{equation}\label{bebe2}
	\int_0^a\frac{k(t) \ln^n t}{(\omega+t)}\,\mathrm{d}t = \sum_{j=0}^{\infty} (-1)^j \omega^j \reglim{\nu}{0} \bbint{0}{a} \frac{k(t)\ln^n t}{t^{j+\nu+1}}\,\mathrm{d}t + k(-\omega) \reglim{\nu}{0}\Delta_n(\nu,\omega).
\end{equation}

First let us evaluate the regularized limit in the second term of \eqref{bebe2}. The form of $\Delta_n(\nu,\omega)$ given by equations \eqref{wa1} and \eqref{wa2} is not convenient to perform the regularized limit on. Instead we use the known identity \cite[pg-9, \#10]{yury}
\begin{equation}
\begin{split}
    \frac{\mathrm{d}^l}{\mathrm{d}\nu^l}\csc(\pi\nu) =& (-1)^{l+1} \frac{l!}{\pi\nu^{l+1}} +\frac{1}{\pi2^{l+1}} \left[\psi^{(l)}\!\left(\frac{1+\nu}{2}\right)-\psi^{(l)}\!\left(\frac{\nu}{2}\right)\right]\\
    &-\frac{(-1)^l}{\pi2^{l+1}} \left[\psi^{(l)}\!\left(\frac{1-\nu}{2}\right)-\psi^{(l)}\!\left(-\frac{\nu}{2}\right)\right]
    \end{split}
\end{equation}
where $\psi^{(n)}(z)$ is the polygamma function of order $n$. Substitution and simplification yield
\begin{equation}
	\begin{split}
	\Delta_n(\nu,\omega) =&\sum_{l=0}^n (-1)^l \binom{n}{l}(\operatorname{Log}\omega)^{n-l}\\
	&\hspace{-6mm}\times\left\{(-1)^{l+1}\,l!\,\frac{ \omega^{-\nu}}{\nu^{l+1}}+ \frac{\omega^{-\nu}}{2^{l+1}}\left[\psi^{(l)}\!\left(\frac{1+\nu}{2}\right)-(-1)^l \psi^{(l)}\!\left(\frac{1-\nu}{2}\right)\right]\right.\\
	& \left.- \frac{\omega^{-\nu}}{2^{l+1}}\left[\psi^{(l)}\!\left(\frac{\nu}{2}\right)-(-1)^l \psi^{(l)}\!\left(\frac{-\nu}{2}\right)\right]\right\} .
	\end{split}
\end{equation}
Again using the linearity of the regularized limit, we have
\begin{equation}
	\begin{split}
		\reglim{\nu}{0}\Delta_n(\nu,\omega) =&\sum_{l=0}^n (-1)^l \binom{n}{l}(\operatorname{Log}\omega)^{n-l}
		\left\{(-1)^{l+1} l!\, \reglim{\nu}{0} \frac{ \omega^{-\nu}}{\nu^{l+1}}\right.\\
		&+\frac{1}{2^{l+1}} \reglim{\nu}{0}\omega^{-\nu}\left[\psi^{(l)}\!\left(\frac{1+\nu}{2}\right)-(-1)^l \psi^{(l)}\!\left(\frac{1-\nu}{2}\right)\right]\\
		& \left.- \frac{1}{2^{l+1}}\reglim{\nu}{0}\omega^{-\nu}\left[\psi^{(l)}\!\left(\frac{\nu}{2}\right)-(-1)^l \psi^{(l)}\!\left(\frac{-\nu}{2}\right)\right]\right\}
	\end{split}
\end{equation}

The first regularized limit can be obtained using direct application of equation \eqref{coco}. The result is
\begin{equation}
	\reglim{\nu}{0} \frac{ \omega^{-\nu}}{\nu^{l+1}} = (-1)^{l+1} \frac{\operatorname{Log}^{l+1}\omega}{(n+1)!}.
\end{equation}
Owing to the analyticity of the polygamma function in the right half plane $\mathrm{Re}(\nu)>0$, the second regularized limit reduces to the Cauchy limit. The limit is proportional to
\begin{equation}
	\psi^{(l)}\!\left(\frac{1}{2}\right) = (-1)^{l+1} l! (2^{l+1}-1) \zeta(l+1),
\end{equation}
where $\zeta(z)$ is the Riemann zeta function. Then we obtain the regularized limit
\begin{equation}
	\reglim{\nu}{0}\omega^{-\nu}\left[\psi^{(l)}\!\left(\frac{1+\nu}{2}\right)-(-1)^l \psi^{(l)}\!\left(\frac{1-\nu}{2}\right)\right]=(1-(-1)^{l}) l! (2^{l+1}-1) \zeta(l+1)
\end{equation}
Observe that only odd $l$ contributes. Now the third limit involves a simple pole at $\nu=0$. The computation of the limit is facilitated by the asymptotic behavior of the polygamma function \cite{polyWolf},
\begin{equation}
	\psi^{(l)}(z)=\frac{(-1)^{l-1} l!}{z^{l+1}} + (-1)^{l-1} l! \zeta(l+1) (1+O(z)),\; z\rightarrow 0,\; l=1, 2, 3, \dots .
\end{equation}
We obtain the limit
\begin{equation}
	\reglim{\nu}{0}\omega^{-\nu}\left[\psi^{(l)}\!\left(\frac{\nu}{2}\right)-(-1)^l \psi^{(l)}\!\left(\frac{-\nu}{2}\right)\right]=\frac{2^{2l+l}}{n+1}(\operatorname{Log}\omega)^{l+1} + (-1)^l l! (1-(-1)^l) \zeta(l+1) .
\end{equation}

Substituting back the regularized limit and gathering all the terms together, we obtain
\begin{equation}
	\Delta_n(\omega)=-\frac{\ln^{n+1}}{(n+1)}+\sum_{l=0}^{n} \binom{n}{l} l! \frac{(2-2^{l+1})}{2^{l+1}}(1-(-1)^l)  \zeta(l+1) (\operatorname{Log}\omega)^{n-l} .
\end{equation}
Only the odd terms in the summation contribute and the contributing terms are proportional to
\begin{equation}
	\zeta(2j) = \frac{(-1)^{j-1} 2^{2j-1}\pi^{2j}}{(2j)!} B_{2n}
\end{equation}
in which $B_{2n}$'s are the Bernoulli numbers. Simplifying the sum to include only the non-vanishing odd terms leads to \eqref{gen2}. 

Now let us consider the infinite series of regularized limits. First we consider the case for $\rho_0<a<\infty$. By Lemma-\ref{mo} the infinite series becomes
\begin{equation}
    \sum_{j=0}^{\infty} (-1)^j \omega^j \reglim{\nu}{0} \bbint{0}{a} \frac{k(t)\ln^n t}{t^{j+\nu+1}}\,\mathrm{d}t = \sum_{j=0}^{\infty} (-1)^j \omega^j \bbint{0}{a} \frac{k(t)\ln^n t}{t^{j+1}}\,\mathrm{d}t
\end{equation}
We now show explicitly that this infinite series converges absolutely. We first obtain a bound for the finite-part integrals. From equation \eqref{xoxo3}, we have the representation for the relevant finite-part integrals,
\begin{equation}\label{nunu}
	\begin{split}
	\bbint{0}{a}\frac{k(t)\ln^n t}{t^{r+1}}\,\mathrm{d}t =& \int_{\epsilon}^a \frac{k(t)\ln^n t}{t^{r+1}}\,\mathrm{d}t + a_r \frac{\ln^{n+1}\epsilon}{(n+1)}\\&+\sum_{j=0}^n \frac{(-1)^j \ln^j\epsilon}{j!} \sum_{l=0,l\neq r}^{\infty} \frac{a_l \epsilon^{l-r}}{(l-r)^{n-j+1}} .
	\end{split}
\end{equation} 
Taking the modulus of both sides of \eqref{nunu}, we obtain the inequality,
 \begin{equation}\label{ineq2}
 	\begin{split}
 		\left|\bbint{0}{a}\frac{k(t)\ln^n t}{t^{r+1}}\,\mathrm{d}t\right| \leq & \frac{1}{\epsilon^r} \int_{\epsilon}^a \frac{\left|k(t)\ln^n t\right|}{t}\,\mathrm{d}t + \left|a_r\right| \frac{\left|\ln\epsilon\right|^{n+1}}{(n+1)}\\&+\frac{1}{\epsilon^r}\sum_{j=0}^n \frac{\left|\ln\epsilon\right|^j}{j!} \sum_{l=0,l\neq r}^{\infty} \frac{\left|a_l\right| \epsilon^{l}}{\left|l-r\right|^{n-j+1}} .
 	\end{split}
 \end{equation} 
where the first terms follows from inequality \eqref{ne} with $\nu=0$. Now we have the bound,
\begin{equation}
	\frac{1}{\left|l-r\right|}\leq 1 ,\;\; l\neq r .
\end{equation}
This translates inequality \eqref{ineq2} into
 \begin{equation}\label{paw}
	\begin{split}
		\left|\bbint{0}{a}\frac{k(t)\ln^n t}{t^{r+1}}\,\mathrm{d}t\right| \leq &  \left|a_r\right| \frac{\left|\ln\epsilon\right|^{n+1}}{(n+1)}+\frac{M_0(a,\epsilon)}{\epsilon^r} ,
	\end{split}
\end{equation} 
where
\begin{equation}
	\begin{split}
		M_0(a,\epsilon) = \int_{\epsilon}^a \frac{\left|k(t)\ln^n t\right|}{t}\,\mathrm{d}t+\sum_{j=0}^n \frac{\left|\ln\epsilon\right|^j}{j!} \sum_{l=0,l\neq r}^{\infty} \left|a_l\right| \epsilon^{l} .
	\end{split}
\end{equation}

We are now ready to prove the absolute convergence of the infinite series in \eqref{kwa}. Taking the modulus of both sides of the series, we have the inequality,
\begin{equation}
	\left|\sum_{r=0}^{\infty}(-1)^r \omega^r \bbint{0}{a}\frac{k(t) \ln^n t}{t^{r+1}}\right|\leq \sum_{r=0}^{\infty}\left|\omega\right|^r \left|\bbint{0}{a}\frac{k(t) \ln^n t}{t^{r+1}}\right| .
\end{equation}
Substituting the bound \eqref{paw} on the finite-part integrals, we arrive at
\begin{equation}
	\left|\sum_{r=0}^{\infty}(-1)^r \omega^r \bbint{0}{a}\frac{k(t) \ln^n t}{t^{r+1}}\right|\leq \frac{\left|\ln\epsilon\right|^{n+1}}{(n+1)} \sum_{r=0}^{\infty}\left|a_r \omega\right|^r +M_0(q,\epsilon)\sum_{r=0}^{\infty} \left|\frac{\omega}{\epsilon}\right|^r .
\end{equation}
The first term converges whenever $|\omega|<\rho_0$ and the second term converges whenever $\omega<\epsilon$. Again for every $\omega$ with $|\omega|<\rho_0$ there always exists a positive $\epsilon$ such that $|\omega|<\epsilon<\rho_0$. Under this condition both terms converge simultaneously for all $|\omega|<\rho_0$.  

For the $a<\rho_0$ case, we use the same method and arguments as in the $\nu\neq 0$ case to show that the infinite series converges under the necessary condition that $|\omega|<a$. 

For the $a=\infty$ case, we use the same arguments as in the $\nu\neq 0$ case to prove that equation \eqref{kwa} holds when $a$ is replaced with infinity. 

The asymptotic relation \eqref{asym2} follows from the fact that
	\begin{equation}\label{kwa2}
	\int_0^a \frac{k(t)\ln^n t}{(\omega+t)}\,\mathrm{d}t=\left(\bbint{0}{a}\frac{k(t) \ln^n t}{t}+O(\omega)\right) + \Delta_n(\omega)\left(k(0)+O(\omega))\right),\;\; \omega\rightarrow 0 .
\end{equation}
Since $\Delta_n(\omega)=O(\operatorname{Log}^{n+1}(\omega))$ as $\omega\rightarrow 0$, the second term dominates the first term and \eqref{asym2} follows.

\end{proof}

\subsubsection*{Example} For $n=1$ equation \eqref{kwa} reduces to the Stieltjes transform
\begin{equation}\label{sum1}
	\int_0^a \frac{k(t)\ln(t)}{(\omega+t)}\,\mathrm{d}t = \sum_{k=0}^{\infty} (-1)^k \omega^k \bbint{0}{a} \frac{k(t) \ln(t)}{t^{k+1}}\,\mathrm{d}t - k(-\omega) \left(\frac{1}{2}\operatorname{Log}^2\omega +\frac{\pi^2}{6} \right),
\end{equation}
true under the conditions of Theorem-\ref{hihi}. The asymptotic relation \eqref{asy2} follows from this expression for arbitrarily small $\omega$.

We can arrive at \eqref{sum1} by explicit finite-part integration using the contour integral representation of the finite-part integral given by \eqref{fpi2c}. Under the same relevant conditions as in Theorem-\ref{miyo}, the Stieltjes transform can be extracted from the contour integral
\begin{equation}
	\frac{1}{2\pi i} \int_{\mathrm{C}} \frac{k(z)}{\omega+z} \left[\frac{1}{2}\log^2z - i \pi \log z -\frac{\pi^2}{3}\right]\,\mathrm{d}z .
\end{equation}
Deforming the contour $C$ into the contour $C'$, the Stieltjes transform takes the representation
\begin{equation}\label{rep22}
	\begin{split}
		\int_0^a \frac{k(t) \ln t}{\omega + t}\,\mathrm{d}t =& \frac{1}{2\pi i} \int_{\mathrm{C}} \frac{k(z)}{\omega+z} \left[\frac{1}{2}\log^2z - i \pi \log z 
		-\frac{\pi^2}{3}\right]\mathrm{d}z\\
		&- k(-\omega) \left(\frac{1}{2}\ln^2\omega + \frac{\pi^2}{6}\right),
	\end{split}
\end{equation} 
assuming $\omega>0$. The second term is again the contribution from the pole of the kernel of transformation at $-\omega$. Expanding the kernel $(\omega+z)^{-1}$ about $\omega=0$ and distributing the integration, we arrive at \eqref{sum1} upon identifying the contour integrals as the finite-part integrals and replacing the multivalued functions with their principal values.

\section{A Combination of Logarithmic Singularities}\label{combination}
We now take up a specific implementation of the finite-part integration of the Stieltjes transform of the form 
\begin{equation}
	\int_0^a \frac{g(t)}{t^{\nu}(\omega +t)}\,\mathrm{d}t ,
\end{equation}
where \begin{equation}\label{sumlog}
	g(t)=\sum_{j=0}^N k_j(t) \ln^j t,
\end{equation}
in which the $k_j(t)$'s belong to $\mathcal{K}_a$. 
This leads to the consideration of the analytic continuation of the Mellin transform
\begin{equation}
	\mathcal{M}_a[g(t);1-\lambda]=\int_0^{a}\frac{g(t)}{t^{\lambda}}\,\mathrm{d}t ,\;\; c<\mathrm{Re}(\lambda)<1, 
\end{equation}
and its relationship to the finite-part integral
\begin{equation}
	\bbint{0}{a} \frac{g(t)}{t^{\lambda}}\,\mathrm{d}t, \;\;\;\mathrm{Re}(\lambda)\geq 1 .
\end{equation}
As in the previous cases, the bound $c$ is determined by the upper limit of integration $a$ and the properties of the $k_j(t)$'s.   

\subsection{Analytic continuation and finite-parts}
Because each $k_j(t) \ln^j t$ belongs to $\mathcal{K}_a$, each Mellin transform $\mathcal{M}_a[k_j(t)\ln^n t;1-\lambda]$ exists separately so that the Mellin transform of $g(t)$ is just the linear sum of the individual transforms. This translates to the same statement on the analytic continuation of the Mellin transform of $g(t)$. 
\begin{theorem} Under the stated conditions on the $k_j(t)$'s,
\begin{equation}\label{sumanal}
	\mathcal{M}_a^*[g(t);1-\lambda]=\sum_{j=0}^N \mathcal{M}_a^*[k_j(t)\ln^j t;1-\lambda] 
\end{equation}
where the analytic continuations in the right hand side are given by equation \eqref{generalanal}. 
\end{theorem}

The analytic structure of the analytic continuation now depends on the specific combination and on the analytic properties of the $k_j(z)$'s. Let us consider the simplest of cases,
\begin{equation}\label{linear}
	h(t)=k_0(t)+k_1(t) \ln t .
\end{equation}
We have the analytic continuation
\begin{equation}
	\mathcal{M}_a^*[h(t);1-\lambda]=\mathcal{M}_a^*[k_0(t);1-\lambda]+\mathcal{M}_a^*[k_1(t)\ln t;1-\lambda].
\end{equation}
We know that $\mathcal{M}_a[k_1(t)\ln t;1-\lambda]$ has a double pole or a removable singularity at positive integers; also we know that $\mathcal{M}_a[k_0(t);1-\lambda]$ has a simple pole or a removable singularity at the same points. If $k_1^{(m-1)}(0)\neq 0$, then $\mathcal{M}_a[h(t);1-\lambda]$ has a double pole at $\lambda=m$. If $k_1^{(m-1)}(0)=0$, the singularity of  $\mathcal{M}_a[h(t);1-\lambda]$ is now determined by the singularity of $\mathcal{M}_a^*[k_0(t);1-\lambda]$. If $k_0^{(m-1)}(0)\neq 0$, then  $\mathcal{M}_a[h(t);1-\lambda]$ has a simple pole at $\lambda=m$; on the other hand,  if  $k_0^{(m-1)}(0)=0$,  $\mathcal{M}_a[h(t);1-\lambda]$ has a removable singularity at $\lambda=0$. Thus $\mathcal{M}_a[h(t);1-\lambda]$ may have a double pole, a simple pole or a removable singularity at positive integers $\lambda=m$ dictated by the zeros of $k_0(t)$ and $k_1(t)$ and their derivatives in the positive real line. 

The same analysis applies for an arbitrary combination of logarithmic singularities at the origin in the form of the function $g(t)$ given by \eqref{sumlog}. We conclude.
\begin{theorem}
	If $k_N^{(m-1)}(0)\neq 0$, then $\mathcal{M}_a^*[h(t);1-\lambda]$ has a pole of order $(N+1)$ at $\lambda=m$. If $k_J^{(m-1)}(0)\neq 0$ and $k_j^{(m-1)}=0$ for all $j>J$, then $\mathcal{M}_a^*[h(t);1-\lambda]$ has a pole of order $(J+1)$ at $\lambda=m$. If $k_j^{(m-1)}(0)=0$ for all $j=0,\dots,N$, then $\lambda=m$ is a regular point or a removable singularity of $\mathcal{M}_a^*[h(t);1-\lambda]$.
\end{theorem} 

Now it is immediate from the definition of the finite-part that the process of extracting the finite-part is linear, owing to the linearity of integration over a finite sum of integrable functions. Thus we have
\begin{equation}
	\bbint{0}{a} \frac{h(t)}{t^{\lambda}}\,\mathrm{d}t =\sum_{j=0}^N \bbint{0}{a} \frac{k_j(t)\ln^j t}{t^{\lambda}}\,\mathrm{d}t .
\end{equation}
This implies the following result.
\begin{theorem} \label{theoremcom}
	At $\lambda\neq 1, 2, 3,\dots$,
	\begin{equation}
	\bbint{0}{a}\frac{g(t)}{t^{\lambda}}\,\mathrm{d}t = \mathcal{M}_a^*[h(t);1-\lambda],
	\end{equation} 
and at $\lambda=m=1, 2, 3,\dots$,
	\begin{equation}
	\bbint{0}{a}\frac{g(t)}{t^{m}}\,\mathrm{d}t =\reglim{\lambda}{m} \mathcal{M}_a^*[h(t);1-\lambda].
\end{equation} 
\end{theorem}

The importance of this result lies on the possibility that the analytic continuations in the right hand side of equation \eqref{sumanal} may not be available for each term but the analytic continuation of the right hand side may be available. Then Theorem-\ref{theoremcom} allows us to evaluate the finite-part from the direct analytic continuation of the Mellin transform of $g(t)$ taken as a whole.

\subsection{Linear Logarithmic Case}
We now evaluate the Stieltjes transform for the specific case of $g(t)$ given by 
\begin{equation}\label{linearcase}
	g(t)=k_0(t)+k_1(t) \ln t,
\end{equation}
where $k_0(t)$ and $k_1(t)$ are both in $\mathcal{K}_a$. Examples of special functions falling under this family of functions are the Bessel functions of the second kind, $K_{\nu}(z)$ and $Y_{\nu}(z)$. Substituting we have
 \begin{equation}\label{bobo}
 	\int_0^{a}\frac{g(t)}{t^{\nu}(\omega+t)}\,\mathrm{d}t = \int_0^a \frac{k_0(t)}{t^{\nu} (\omega+t)}\,\mathrm{d}t + \int_0^a \frac{k_1(t) \ln t}{t^{\nu} (\omega+t)}\,\mathrm{d}t
 \end{equation}
The Stieltjes transform for the non-logarithmic case is already known and are given by equation \eqref{orig1} and \eqref{orig2}, and for the linear case are given by equations \eqref{sum2} and \eqref{sum1}. Substituting them back into equation \eqref{bobo}, we arrive at the following results.
\begin{proposition}\label{weh} 
	Given the conditions on $g(t)$ of \eqref{linearcase}. Let $\rho_0^{(0)}$ and $\rho_0^{(1)}$ be the respective distances of the singularities of $k_0(z)$ and $k_1(z)$ nearest to the origin. Then  
\begin{equation}\label{koko}
	\begin{split}
	\int_0^a \frac{g(t)}{t^{\nu} (\omega+t)}\,\mathrm{d}t =& \sum_{j=0}^{\infty} (-1)^j \omega^j \bbint{0}{a} \frac{g(t)}{t^{j+\nu+1}}\,\mathrm{d}t\\
	& \hspace{-20mm}+ \frac{\pi}{\omega^{\nu} \sin(\pi\nu)} \left[k_0(-\omega)+k_1(-\omega)\operatorname{Log}\omega + \pi k_1(-\omega)\cot(\pi\nu)\right],\; \nu\neq 0 .
	\end{split}
\end{equation}
for all $|\omega|<\mathrm{min}(\rho_0^{(0)},\rho_0^{(1)},a)$ with $|\mathrm{Arg}\,\omega|<\pi$. If it happens that $k_0(t)$ and $k_1(t)$ are both even or odd, then \eqref{koko} simplifies to
\begin{equation}\label{bwa}
	\begin{split}
		\int_0^a \frac{g(t)}{t^{\nu} (\omega+t)}\,\mathrm{d}t =& \sum_{j=0}^{\infty} (-1)^j \omega^j \bbint{0}{a} \frac{g(t)}{t^{j+\nu+1}}\,\mathrm{d}t\\
		& \hspace{-20mm}\pm \frac{\pi}{\omega^{\nu} \sin(\pi\nu)} \left[g(\omega) + \pi k_1(\omega)\cot(\pi\nu)\right],\; \nu\neq 0 ,
	\end{split}
\end{equation}
where the upper (lower) sign holds when the functions are both even (odd).

\end{proposition}
\begin{proposition} Under the same relevant conditions as in Proposition-\ref{weh},
\begin{equation}\label{hew}
    \int_0^a \frac{g(t)}{\omega+t}\,\mathrm{d}t =  \sum_{j=0}^{\infty} (-1)^j \omega^j \bbint{0}{a}\frac{g(t)}{t^{j+1}}\,\mathrm{d}t - k_0(-\omega) \operatorname{Log}\omega - k_1(-\omega) \left(\frac{\operatorname{Log}^2\omega}{2}+\frac{\pi^2}{6}\right) .
\end{equation}
If it happens that $k_0(t)$ and $k_1(t)$ are both even or odd, then \eqref{hew} simplifies to
\begin{equation}\label{linear2}
    \int_0^a \frac{g(t)}{\omega+t}\,\mathrm{d}t =  \sum_{j=0}^{\infty} (-1)^j \omega^j \bbint{0}{a}\frac{g(t)}{t^{j+1}}\,\mathrm{d}t \mp g(\omega) \operatorname{Log}\omega \pm k_1(\omega) \left(\frac{\operatorname{Log}^2\omega}{2}+\frac{\pi^2}{6}\right) ,
\end{equation}
where the upper (lower) signs hold when the functions are both even (odd).  
\end{proposition}

\subsection{Example}
As an example, we evaluate the following Stieltjes transform using finite-part integration,
\begin{equation}
    \int_0^{\infty} \frac{Y_0(t)}{t^{\nu}(\omega+t)}\,\mathrm{d}t,
\end{equation}
where $Y_0(t)$ is a Bessel function of the second kind. This Bessel function has the representation
\begin{equation}
    Y_0(t)=\frac{2}{\pi} J_0(t) \ln t-\frac{2}{\pi} J_0(t) \ln 2  -\frac{2}{\pi}\sum_{k=0}^{\infty} \frac{(-1)^k}{(k!)^2} \psi(k+1) \left(\frac{t}{2}\right)^{2k},
\end{equation}
where $J_0(t)$ is a Bessel function of the first kind. This fall under \eqref{linearcase} with
\begin{equation}
    k_1(t)=\frac{2}{\pi} J_0(t)
\end{equation}
\begin{equation}
    k_0(t)=-\frac{2}{\pi} J_0(t) \ln 2 -\frac{2}{\pi}\sum_{k=0}^{\infty} \frac{(-1)^k}{(k!)^2} \psi(k+1) \left(\frac{t}{2}\right)^{2k} .
\end{equation}
Both $k_0(z)$ and $k_1(z)$ are entire and even in their variables. 

The problem reduces to evaluating the finite-part integrals. We use tabulated Mellin transform and invoke analytic continuation. We have the known result \cite[pg. 172, \#3.11.1.1]{yu},
\begin{equation}\label{mellinyu}
    \int_0^{\infty} t^{-\lambda} Y_0(t)\,\mathrm{d}t = -\frac{1}{\pi 2^{\lambda}} \sin\!\left(\frac{\lambda\pi}{2}\right) \left[\Gamma\!\left(\frac{1-\lambda}{2}\right)\right]^2, \;\;\; -\frac{1}{2}<\mathrm{Re}(\lambda)<1 .
\end{equation}
The right hand side is analytic everywhere except at the points $\lambda=1, 3, 5, \dots$, and gives the desired analytic continuation of the Mellin transform,
\begin{equation}
\mathcal{M}^*[Y_0(t);1-\lambda]    =-\frac{\pi\sin\!\left(\frac{\lambda\pi}{2}\right)}{ 2^{\lambda} \cos^2\left(\frac{\pi\lambda}{2}\right) \left[\Gamma\!\left(\frac{1+\lambda}{2}\right)\right]^2}  
\end{equation}
where the right hand side has been arrived at by applying the reflection formula to the gamma function,
\begin{equation}\label{reflection}
	\Gamma(-z)=\frac{\pi \csc(\pi z)}{\Gamma(z+1)},
\end{equation}
 in equation \eqref{mellinyu}. The analytic continuation has zeros at $\lambda =2, 4, 6, \dots$ so that the finite-part integral at these points is zero. The analytic continuation has double pole at odd positive integers. 

\subsubsection{$\nu\neq 0$}
Equation \eqref{bwa} evaluates the Stieltjes integral into
\begin{equation}\label{xo}
	\begin{split}
	\int_0^{\infty} \frac{Y_0(t)}{t^{\nu}(\omega+t)}\,\mathrm{d}t =& \sum_{j=0}^{\infty} (-1)^j \omega^j \bbint{0}{a} \frac{Y_0(t)}{t^{j+\nu+1}}\,\mathrm{d}t \\
	&+ \frac{\pi}{\omega^{\nu} \sin(\pi\nu)} \left[Y_o(\omega) + 2 J_0(\omega) \cot(\pi\nu)\right].
	\end{split}
\end{equation}
The finite-part integrals correspond to the points $\lambda=j+\nu+1$ which are points of analyticity of the analytic continuation because $(j+\nu+1)$ is non-integer given that $j$ is an integer and $\nu$ is a non-integer. Then the finite-parts coincide with the values of the analytic continuation at those points, 
\begin{equation}
	\begin{split}
	\bbint{0}{\infty} \frac{Y_0(t)}{t^{j+\nu+1}}\,\mathrm{d}t = - \frac{\pi \sin\!\left(\frac{\pi}{2}(j+1+\nu)\right)}{2^{j+1+\nu} \cos^2\!\left(\frac{\pi}{2}(j+1+\nu)\right)\left[\Gamma\!\left(\frac{1}{2}(j+1+\nu)\right)\right]^2}.
	\end{split}
\end{equation}
The finite-part integrals split in two groups depending on whether $j$ is even or odd. For the even terms, $j=2 k$, $k=0, 1, 2,\dots$, we have the values
\begin{equation}\label{f1}
	\begin{split}
	\bbint{0}{\infty} \frac{Y_0(t)}{t^{2k+\nu+1}}\,\mathrm{d}t = \frac{\pi \cos(\pi\nu/2) \csc^2(\pi\nu/2)}{2^{\nu+1}\left[\Gamma(\nu/2)\right]^2} \frac{(-1)^{k+1}}{2^{2k} \left(\nu/2\right)_k \left(\nu/2\right)_k}
\end{split}
\end{equation}
For the odd terms, $j=2k+1$, $k=0,1,2,\dots$, we have likewise the values
\begin{equation}\label{f2}
	\begin{split}
		\bbint{0}{\infty} \frac{Y_0(t)}{t^{2k+\nu+2}}\,\mathrm{d}t = \frac{\pi \sin(\pi\nu/2) \sec^2(\pi\nu/2)}{2^{\nu+2}\left[\Gamma(3/2+\nu/2)\right]^2} \frac{(-1)^{k}}{2^{2k} \left(3/2+\nu/2\right)_k \left(3/2+\nu/2\right)_k} .
	\end{split}
\end{equation}

We substitute the finite-part integrals \eqref{f1} and \eqref{f2} back into  equation \eqref{xo}, sum separately the even and odd terms, and simplify using the reflection formula \eqref{reflection}. We obtain the result
\begin{equation}\label{neni}
	\begin{split}
	\int_0^{\infty}\frac{Y_0(t)}{t^{\nu}(\omega+t)}\,\mathrm{d}t =&\frac{\pi}{\omega^{\nu} \sin(\pi\nu)} \left[Y_0(\omega) + 2 J_0(\omega) \cot(\pi\nu)\right] \\ &\hspace{-24mm}-\frac{\cos(\pi\nu/2) \Gamma^2(-\nu/2)}{\pi 2^{\nu+1}} \hypergeom{1}{2}{1}{1+\nu/2 , 1+\nu/2}{-\frac{\omega^2}{4}}\\
	&\hspace{-24mm}-\frac{\sin(\pi\nu/2) \Gamma^2(-1/2-\nu/2)}{\pi 2^{\nu+2}} \hypergeom{1}{2}{1}{3/2+\nu/2 , 3/2+\nu/2}{-\frac{\omega^2}{4}},
	\end{split}
\end{equation}
which is valid for $0<\mathrm{Re}(\nu)<1$ and for all $|\omega|<\infty$, $|\mathrm{Arg}(\omega)|<\pi$. The second term arises from the even term contributions; and the third term, from the odd term contributions. The result is valid for all non-negative $\omega$ because $k_0(z)$ and $k_1(z)$ are both entire, and the range of integration is the entire positive real line.

\subsubsection{$\nu=0$}
Equation \eqref{linear2} evaluates the Stieltjes transform into
\begin{equation}\label{pipi}
	\int_0^{\infty}\frac{Y_0(t)}{\omega+t}\,\mathrm{d}t =\sum_{j=0}^{\infty} (-1)^j \omega^j \bbint{0}{\infty} \frac{Y_0(t)}{t^{j+1}}\,\mathrm{d}t - Y_0(\omega) \operatorname{Log}\omega + \frac{J_0(\omega)}{\pi}\left(\frac{\operatorname{Log}^2\omega}{2}-\frac{\pi^2}{6}\right) .
\end{equation}
We have seen earlier that the analytic continuation of the Mellin transform vanish at positive even $\lambda$'s so that the terms $j=1,3,5,\dots$ do not contribute, and only the terms $j=0, 2, 4,\dots$ contribute. The contributing finite-parts are given by
\begin{equation}
	\bbint{0}{\infty}\frac{Y_0(t)}{t^{2l+1}}\,\mathrm{d}t = \reglim{\lambda}{2l+1} \mathcal{M}^*[Y_0(t);1-\lambda], \;l=0,1,2,\dots .
\end{equation}

To evaluate the non-vanishing finite-part integrals, we rationalize the analytic continuation in the form
\begin{equation}
	\mathcal{M}^*[Y_0(t);1-\lambda]=\frac{f(\lambda)}{g(\lambda)},
\end{equation} where 
\begin{equation}
f(\lambda)=-\frac{\pi\sin\!\left(\frac{\lambda\pi}{2}\right)}{ 2^{\lambda} \cos^2\left(\frac{\pi\lambda}{2}\right) \left[\Gamma\!\left(\frac{1+\lambda}{2}\right)\right]^2},\;\;\;g(\lambda)=\cos^2(\pi\lambda/2).
\end{equation} 
Since the pole is order 2, the regularized limit is given by equation \eqref{reglimdouble}. But $g'''(2l+1)=0$ so that the regularized limit simplifies to
\begin{equation}\label{reglimdouble2}
	\begin{split}
	\reglim{\lambda}{2l+1} \frac{f(\lambda)}{g(\lambda)} =& \frac{f''(2l+1)}{g''(2l+1)}  -\frac{2}{3} \frac{f'(2l+1) g'''(2l+1)}{(g''(2l+1))^2}\\
	 &- \frac{f(2l+1) g''(2l+1) g''''(2l+1)}{6 (g''(2l+1))^3} .
	\end{split}
\end{equation}
Evaluating the indicated derivatives at $\lambda=2l+1$, we obtain the desired finite-part integrals,
\begin{equation}\label{yo}
\begin{split}
    \bbint{0}{\infty} \frac{Y_0(t)}{t^{2l+1}}\,\mathrm{d}t =& \frac{(-1)^l}{3\pi 2^{2(l+1)} (l!)^2}\\ &\hspace{-18mm}\times\left[\pi^2-12 \ln^2 2 - 24 \ln 2 \psi(l+1) - 12 (\psi(l+1))^2 + 6 \psi^{(1)}(l+1)\right].
    \end{split}
\end{equation}

Substituting the finite-part integrals \eqref{yo} back into equation \eqref{pipi} and performing some simplifications, we obtain the Stieltjes transform
\begin{equation}
\begin{split}
    \int_0^{\infty} \frac{Y_0(t)}{\omega+t}\,\mathrm{d}t =& \frac{J_0(\omega)}{\pi}\left(\operatorname{Log}^2 \omega -\ln^22 -\frac{\pi^2}{4}\right) - Y_0(\omega) \operatorname{Log}\omega\\ &\hspace{-20mm}+\frac{2}{\pi}\sum_{l=0}^{\infty} \frac{(-1)^l}{(l!)^2} \left[\psi^{(1)}(l+1)-2 (\psi(l+1))^2 - 4\ln 2 \, \psi(l+1)\right]\left(\frac{\omega}{2}\right)^{2l},
    \end{split}
\end{equation}
which is valid for all $|\omega|<\infty$, $|\operatorname{Arg}\omega|<\pi$, for the same reasons as in \eqref{neni}.

\section*{Acknowledgment}
The author acknowledges the kind and generous invitation of Prof. Roderick Wong to Liu Bie Ju Center of Mathematical Sciences of the City University of Hong Kong in 2019. It was an encouraging experience for the author to discuss finite-part integration with Prof. Wong, and with his students and collaborators. It was enlightening  to the author to personally listen to Prof. Wong discuss his distributional approach to asymptotics. The author also acknowledges the ensuing insightful conversations with Dr. Yutian Li, and the warm hospitality of Ms. Sophie Xie. The idea of regularized limit was conceived and formalized during the author's visit at the Center. This work was funded by the UP-System Enhanced Creative Work and Research Grant (ECWRG 2019-05-R).

\end{document}